\documentclass[12pt,reqno]{amsart}

\usepackage{eucal,color,graphicx,subfigure,mathrsfs}

\numberwithin{equation}{section}

\newtheorem{theorem}{Theorem}[section]
\newtheorem{lemma}[theorem]{Lemma}

\theoremstyle{definition}
\newtheorem{definition}[theorem]{Definition}

\theoremstyle{remark}
\newtheorem{remark}[theorem]{Remark}
\newtheorem{assumption}{Assumption}[section]

\numberwithin{equation}{section}

\newcommand{\bx}{\boldsymbol{x}}
\newcommand{\by}{\boldsymbol{y}}
\newcommand{\bp}{\boldsymbol{p}}
\newcommand{\bA}{\boldsymbol{A}}
\newcommand{\bB}{\boldsymbol{B}}
\newcommand{\bK}{\boldsymbol{K}}
\newcommand{\bS}{\boldsymbol{S}}
\newcommand{\bQ}{\boldsymbol{Q}}
\newcommand{\ba}{\boldsymbol{a}}

\begin{document}

\title{ a note on the Monge-Amp{\`{e}}re  type equations with general source terms}

\author{Weifeng Qiu}
\address{Department of Mathematics, City University of Hong Kong,
        83 Tat Chee Avenue, Hong Kong, China} 
\email{weifeqiu@cityu.edu.hk}

\author{Lan Tang}
\address{School of Mathematics and Statistics, Central China Normal University, 430079 Wuhan, 
Hubei, People’s Republic of China}
\email{lantang@mail.ccnu.edu.cn}

\thanks{Weifeng Qiu is partially supported by a grant from the Research Grants 
Council of the Hong Kong Special Administrative Region, China (Project No. CityU 11302014).}

\begin{abstract}In this paper we  consider  the generalised solutions to the Monge-Amp{\`{e}}re  type equations with general source terms.  We firstly prove the so-called comparison principle   and then give some important propositions for the border of generalised solutions. Furthermore, we design well-posed finite element methods  for the generalised solutions with the  classical and weak Dirichlet boundary conditions respectively. 

\end{abstract}



\maketitle

\section{Introduction}
Let $\Omega$  be a bounded open  convex domain in $\mathbb{R}^{d}$  and $W^{+}(\Omega)$ denote the set of convex functions over $\Omega$. Suppose that $\mu$ is a non-negative Borel measure in $\Omega$ and  $R$  is a locally integrable function in $\mathbb{R}^{d}$ with $R(\bp)>0$ for any $\bp \in \mathbb{R}^{d}$. 

In this article,  we mainly consider the generalised solutions to the  Monge-Amp{\`{e}}re  type equations with Dirichlet boundary conditions. Firstly,  for  the   Monge-Amp{\`{e}}re  type equations with the classical Dirichlet boundary condition, the generalised solution is defined as follows:

\begin{definition}
\label{def_ma_measure_eqs}
\textit{We call $u\in W^{+}(\Omega)\cap C(\overline{\Omega})$  a generalised solution to the classical Dirichlet problem of the Monge-Amp{\`{e}}re equation if  the following conditions hold:
\begin{subequations}
\label{ma_measure_eqs}
\begin{align}
\label{ma_measure_eq1}
\int_{\partial u(e)}R(\bp)d\bp & \hspace{10pt}= \hspace{10pt}\mu (e)\hspace{10pt}\text{for\ any Borel set } \ e\subset \Omega,\\
\label{ma_measure_eq2}
u & \hspace{10pt}=\hspace{10pt} g\hspace{10pt} \text{ on } \ \partial\Omega.
\end{align}
\end{subequations}
Here $\partial u(e)$ denotes the sub-differential of $u$ over $e$.}
\end{definition}

Furthermore,  we also consider the Dirichlet problem  with the  weak boundary condition. In this case,  we define generalised solution in the following way: 
\begin{definition}
\label{def_weak_eqs}
\textit{$u\in W^{+}(\Omega)$ is called a generalised solution  to the Monge-Ampere equation with  weak Dirichlet boundary condition, if it  satisfies
\begin{subequations}
\label{weak_eqs}
\begin{align}
\label{weak_eq1}
\displaystyle\int_{\partial u(e)}R(\bp)d\bp &\hspace{10pt} =\hspace{10pt} \displaystyle\mu (e)\quad \hspace{10pt} \text{ for any Borel set } \ e\subset\Omega,\\
\label{weak_eq2}
\limsup_{\Omega\ni\bx^{\prime}\rightarrow \bx} u(\bx^{\prime}) & \hspace{10pt}\leq \hspace{10pt}g(\bx)
\quad\hspace{10pt} \text{for any}\  \bx\in \partial\Omega,
\end{align}
\end{subequations}
and for any $v\in W^{+}(\Omega)$ satisfying (\ref{weak_eqs}),  it holds:
\begin{align}
\label{weak_eq3}
u(\bx) \hspace{10pt}\geq \hspace{10pt}v(\bx)\quad \text{for all}\  \bx \in \Omega.
\end{align}}
\end{definition}
\vspace{5pt}

The Dirichlet problems  (1.1) and (1.2) mentioned above  are natural generalisations of the following problem:
\begin{eqnarray} \left\{\begin{array}{lll}
\mathrm{det} (D^2 u)&=& \mu\ \ \  \text{in } \  \Omega,\\ \\
 \ \hspace{30pt} u&=&g\hspace{10pt} \text{ on } \ \partial\Omega.
\end{array}
\right.
\end{eqnarray}
The boundary problem (1.4), arising from analysis and geometry, plays a very important role in the area of PDEs, and it has received considerable study since 1950's.  This class of problems  were firstly solved in the generalised sense by Alexandrov\cite{Alexandrov58}and Bakelman\cite{Bakelman57}, where   they defined the generalised solution in the same  way as Definition 1.1 and they proved the existence and uniqueness of generalised solutions (see also \cite{Gutierrez01}). 

With additional assumptions that $d\mu= fdx$   and  $0<f\in C(\overline{\Omega})$,  Caffarelli \cite{Caffarelli90b} showed the equivalence between the notions of   generalised solutions and viscosity solutions (also see \cite{Gutierrez01}). Furthermore, for this special case,  there are various  regularity results  on generalised solutions  if $f$ , $\partial{\Omega}$ and  $g$  are certain regular. For the global regularity, the results were established by Cheng and Yau\cite{CY76, CY77},  Ivochkina\cite{Ivochkina80}, Krylov\cite{Krylov82, Krylov83a, Krylov83b}, Caffarelli, Nirenberg and Spruck\cite{CNS84}, Wang\cite{Wang96}, Trudinger-Wang\cite{TW08} and Savin\cite{Savin13}.  As for the interior regularity of generalised solutions,  we refer readers to the work of Caffarelli\cite{ Caffarelli90a, Caffarelli90b}, De Phillipis and Figalli\cite{DF13} and De Phillipis, Figalli and Savin\cite{DFS13}.

For general $R$ and $\mu$,   due to the lack of the regularity assumptions, it is impossible to relate generalised solutions to viscosity solutions and thus in the general case,  the study on  generalised solutions is totally different from that on viscosity solutions. Especially, the regularity problem on generalised solutions becomes tricky. For boundary problems problems (1.1) and (1.2), the solvability was firstly studied by Bakelman (see \cite{Bakelman94}). In his work, the results on  existence  and uniqueness of generalised solutions to (1.1) and (1.2) were established with more suitable assumptions on  $R$ and $\mu$.

In our work, the main goal  is to study the properties of generalised solutions  to the Dirichlet problems (1.1) and (1.2).  More precisely, we organise the remaining content of this article  as follows:  

Section 2 is devoted to proving  Theorem 2.1, the so-called comparison principle,  an important tool  for studying the generalised Monge-Ampere type equations. Theorem 2.1 is a generalisation of \cite[Theorem~$10.1$]{Bakelman94} and our proof is also inspired by the work of Bakelman. However, at some crucial steps, there are some gaps in the proof of \cite[Theorem~$10.1$]{Bakelman94}( we shall explain these in details in section 2). To overcome this difficulty, two important lemmas: Lemma 2.2 and Lemma 2.3 would be given to derive Theorem 2.1.

In section 3,  we consider boundary behaviour of convex function over $\Omega$. We give  Lemma 3.2, which describes some important proposition on the border of convex functions.

In section 4, we show  convergence properties of a sequence of convex functions and this section consists of  two parts: convergence of a sequence of convex functions inside convex domain and convergence of a sequence of borders of convex functions.

In section 5 and 6, a finite element method to the problem (1.1) would be given and shown to be well-posed. Furthermore, we prove that this finite element method converges to the exact solution to (1.1).

In section 7,  a well-posed finite element method would be designed for (1.2) and we prove the convergence of such method to the exact solution to (1.2).
\vspace{35pt}

\section{Comparison principle}
\label{sec_comparison}

In this section, we prove  the comparison principle:
\begin{theorem}
\label{thm_comparison}
Assume that  $z_{1}$ and $z_{2}\in W^{+}(\Omega)$ satisfy the conditions: $z_{1}\in C^{0}(\overline{\Omega})$ and $\displaystyle z_{1}(\bx) \geq \limsup_{\bx^{\prime}\rightarrow \bx, \bx^{\prime}\in \Omega} z_{2}(\bx)$ for any $ \bx \in \partial\Omega. $
If the following inequality holds:
\begin{align*}
\int_{\partial z_{1}(e)}R(\bp)d\bp\  \leq \ \int_{\partial z_{2}(e)}R(\bp)d\bp \quad 
\text{ for every Borel set } e \subset \Omega,
\end{align*}
then
$z_{1}(\bx) \ \geq \ z_{2}(\bx) , \forall \bx\in \Omega.
$
\end{theorem}{\bf Remark:} Theorem~\ref{thm_comparison} is important not only for showing  uniqueness of the solution 
but also for proving existence of the generalized solution with weak Dirichlet boundary condition
in Definition~\ref{def_weak_eqs} (see (\ref{monotone_decay})).  There are several important facts we need to point out:\

({\bf 1}) For the classical Monge-Ampere equations (where $R\equiv1$), the comparison principle can be proved by the Brunn-Minkowski inequality (see [4, Theorem 1.4.6]). However, for general $0<R\in L^1_{\mathrm{loc}}(\mathbb{R}^d)$, that inequality is not available . Thus the argument of the proof of comparison principle for classical case cannot be applied to the general case.
\

({\bf 2}) {Theorem~\ref{thm_comparison} is basically the same as \cite[Theorem~$10.1$]{Bakelman94}, 
and our proof is inspired by that of \cite[Theorem~$10.1$]{Bakelman94}.However, the proof of \cite[Lemma~$10.2$]{Bakelman94} lacks of some essential details.  In the following  we shall explain those gaps:}

Firstly, {one essential argument in the proof of \cite[Lemma~$10.2$]{Bakelman94} states as follows:  assume  that $z_{1}, z_{2}\in W^{+}(\Omega)$, $Q$ is an open set satisfying $\overline{Q}\subset \Omega$.
If there is $\bp\in\mathbb{R}^{d}$ such that 
$\bp$ is contained in the boundary of $\partial z_{2}(Q)$ , 
$\bp \in \text{Int}(\partial z_{1}(Q))$, and $\partial z_{2}(Q)\subset \partial z_{1}(Q)$, then 
\begin{align}
\label{critical_ineq1}
\int_{\partial z_{1}(Q)} R(\bp) d\bp > \int_{\partial z_{2}(Q)} R(\bp) d\bp.
\end{align}
Since $Q$ is open, $\partial z_{1}(Q)$ and $\partial z_{2}(Q)$ are Lebesgue 
measurable in $\mathbb{R}^{d}$, due to \cite[Property ($D$)]{Bakelman94}. 
If we let $\partial z_{1}(Q)$  be the unit open ball in $\mathbb{R}^{d}$, and 
$\partial z_{2}(Q)$ be $\partial z_{1}(Q)$ minus any $(d-1)$-dimensional subset $P^{\prime}$ 
satisfying $\overline{P^{\prime}} \subset \partial z_{1}(Q)$. For any  $\bp \in \partial P^{\prime}$ , then $\bp$ is contained in the boundary of $\partial z_{2}(Q)$ and 
$\bp \in \text{Int}(\partial z_{1}(Q))$. But it is easy to see that (\ref{critical_ineq1}) does not hold 
for this case.  On the other hand, due to \cite[Property ($B$)]{Bakelman94}, 
 $\partial z_{1}(\overline{Q})$ and $\partial z_{2}(\overline{Q})$ are closed. 
Thus, if there is $\bp\in\mathbb{R}^{d}$ such that 
$\bp$ is contained in the boundary of $\partial z_{2}(\overline{Q})$,
$\bp \in \text{Int}(\partial z_{1}(\overline{Q}))$, and $\partial z_{2}(\overline{Q})\subset \partial z_{1}(\overline{Q})$, then 
\begin{align}
\label{critical_ineq2}
\int_{\partial z_{1}(\overline{Q})} R(\bp) d\bp > \int_{\partial z_{2}(\overline{Q})} R(\bp) d\bp.
\end{align}
Thus in our strategy, to get Lemma~\ref{lemma_comparision_aux2}, a revised form of \cite[Lemma~$10.2$]{Bakelman94}, we shall prove  (\ref{critical_ineq2}) instead of (\ref{critical_ineq1}) . 

Secondly, {our Lemma~\ref{lemma_comparision_aux2} is significantly different from \cite[Lemma~$10.2$]{Bakelman94} 
that we only consider a point $\bx_{0}\in\partial \overline{Q}$ such that $z_{2}$ is differentiable at $\bx_{0}$.  
Then $\partial z_{2}(\bx_{0})$ is just a single point in $\mathbb{R}^{d}$, which would simplify the analysis a lot. 
}
\vspace{10pt}

The proof of Theorem \ref{thm_comparison} would depend on the following two lemmas:

\begin{lemma}
\label{lemma_comparision_aux1}
Let $z\in W^{+}(\Omega)$ and $Q$ be an open 
subset of $\Omega$ with $\overline{Q}\subset \Omega$. We assume that $T$ is a hyperplane 
whose equation is 
\begin{align*}
z = z_{T} + \bp_{T}\cdot (\bx - \bx_{T}), 
\end{align*}
where $\bp_{T}, \bx_{T}\in \mathbb{R}^{d}$, and $z_{T}\in \mathbb{R}$. 
In addition, we suppose
\begin{subequations}
\label{ineqs_lemma_comparision_aux1}
\begin{align}
\label{ineq1_lemma_comparison_aux1}
& z(\bx) \geq z_{T}+\bp_{T}\cdot (\bx - \bx_{T}),\qquad \forall \bx \in \overline{Q},\\
\label{ineq2_lemma_comparison_aux1}
& \exists \bx_{0}\in \partial \overline{Q}\text{ such that } (\bx_{0}, z(\bx_{0}))\in T 
\text{ and } \bp_{T} \notin \partial z(\bx_{0}).
\end{align}
\end{subequations}
Then, $\bp_{T}\notin \partial z(\overline{Q})$. 
\end{lemma}

\begin{proof}
We prove it by contradiction. Assume $\bp_{T}\in \partial z (\overline{Q})$, then 
$\exists \ \bx_{Q}\in \overline{Q}$ such that it holds: 
\begin{align*}
z(\bx) \geq z(\bx_{Q}) + \bp_{T}\cdot (\bx - \bx_{Q}),\quad \forall \bx \in \Omega. 
\end{align*}
By (2.3b), we know that 
$\exists\  \bx_{1}\in\Omega$ such that $$z(\bx_{1}) < z(\bx_{0})+\bp_{T}\cdot (\bx_{1} - \bx_{0}).$$
Combining the two estimates above, we infer
\begin{align*}
 z(\bx_{Q}) + \bp_{T}\cdot (\bx_{1} - \bx_{Q})\leq z(\bx_{1}) < z(\bx_{0}) + \bp_{T}\cdot (\bx_{1} - \bx_{0})\end{align*}which indicates that
$$z(\bx_{0}) > z(\bx_{Q}) + \bp_{T}\cdot (\bx_{0} - \bx_{Q}).$$
Apply (2.3a) to the latest inequality, we arrive at $  z(\bx_{0}) >z_{T}+\bp_{T}\cdot (\bx_{0} - \bx_{T})$,
which contradicts with the fact: $(\bx_0,  z(\bx_0)) \in T$ from (\ref{ineq2_lemma_comparison_aux1}).
\end{proof}

The following Lemma 2.3 is a revised version of \cite[Lemma $10.2$]{Bakelman94}:
\begin{lemma}
\label{lemma_comparision_aux2}
Let $z_{1},z_{2}\in W^{+}(\Omega)$ and $Q$ be any open subset of $\Omega$ such that 
the following conditions hold:
$
\overline{Q}\subset \Omega,  z_{1}<z_{2} \text{ in } Q  \ \text{and} \  z_{1}= z_{2}\text{ on }\partial Q. 
$
Assume that for any $\bx_{Q}\in\partial \overline{Q}$,  $\exists r>0$ such that 
$\overline{B_{r}(\bx_{Q})} \subset \Omega$ and $z_{1}\geq z_{2}$ in $B_{r}(\bx_{Q})\backslash \overline{Q}$. 

If there exists some point $\bx_{0}\in \partial \overline{Q}$ such that $z_{2}$ is differentiable at $\bx_{0}$ 
and $\nabla z_{2}(\bx_{0})\notin \partial z_{1}(\bx_{0})$, then it holds:
\begin{align*}
\int_{\partial z_{1}(\overline{Q})}R(\bp) d\bp > \int_{\partial z_{2}(\overline{Q})}R(\bp) d\bp.
\end{align*}
\end{lemma}

\begin{proof}
By the assumptions it is easy to see that $\partial z_{2}(Q)\subset \partial z_{1}(Q)$.  

For any $\bx_{Q}\in \partial \overline{Q}$, let $T^{\prime}$ be a supporting hyperplane of 
the graph of $z_{2}$ at $(\bx_{Q}, z_{2}(\bx_{Q}))$. Then the equation of $T^{\prime}$ is 
$y = z_{2}(\bx_{Q}) + \bp_{T^{\prime}}\cdot (\bx - \bx_{Q})$ with  some $\bp_{T^{\prime}}\in \partial z_{2}(\bx_{Q})$. 
Thus $\exists r>0$ such that the following holds:
\begin{align*}
z_{1}(\bx) \geq z_{2}(\bx) \geq L_{T^{\prime}}(\bx) := z_{2}(\bx_{Q}) + \bp_{T^{\prime}}\cdot (\bx - \bx_{Q}),
\qquad \forall \bx \in B_{r}(\bx_{Q})\backslash \overline{Q}. 
\end{align*}

If $\bp_{T^{\prime}}\notin \partial z_{1}(\bx_{Q})$, then by the above inequality and the condition that 
$z_{1} = z_{2}$ on $\partial Q$, there exists $\bx_{1}\in Q$ such that $z_{1}(\bx_{1})< L_{T^{\prime}}(\bx_{1})$. 
Thus there exists a new hyperplane parallel to $T^{\prime}$  such that it supports the graph of $z_{1}$ at 
some point $(\bx_{2}, z_{1}(\bx_{2}))$ with $\bx_{2}\in Q$. Hence $\bp_{T^{\prime}}\in \partial z_{1}(Q)$. 
Furthermore we claim that 
\begin{align*}\bp_{T^{\prime}}\in \text{Int}(\partial z_{1}(Q)).
\end{align*}
In fact, if we choose $\epsilon > 0$ and $\bp_{\epsilon}\in B_{\epsilon}(\bp_{T^{\prime}})\subset \mathbb{R}^{d}$, then 
$
L_{T^{\prime}}^{\epsilon} := z_{1}(\bx_{Q}) + \bp_{\epsilon}\cdot (\bx - \bx_{Q})
$
satisfies the following:
\begin{align*}
L_{T^{\prime}}^{\epsilon}(\bx) \geq L_{T^{\prime}}(\bx) - \epsilon\big(\sup_{\bx,\bx^{\prime}\in Q}\vert \bx 
- \bx^{\prime} \vert\big), \quad \forall\bx\in Q.
\end{align*}
Then
$L_{T^{\prime}}^{\epsilon} (\bx_{1}) > z_{1}(\bx_{1})$  for $\epsilon \text{ small enough}.$
Thus there is a supporting hyperplane of the graph of $z_{1}$ at some point $(\bx_{\epsilon}, z_{1}(\bx_{\epsilon}))$ 
with $\bx_{\epsilon}\in Q$, and its equation is 
$
y = z_{1}(\bx_{\epsilon}) + \bp_{\epsilon}\cdot (\bx - \bx_{\epsilon}). 
$
Therefore  $\bp_{\epsilon}\in \partial z_{1}(Q)$ for $ \epsilon \text{ small enough},$ which concludes the claim.

Hence for any $\bx_{Q}\in \partial \overline{Q}$, it holds:
$\partial z_{2}(\bx_{Q})\backslash \partial z_{1}(\bx_{Q}) \subset \text{Int}(\partial z_{1}(Q))$ and 
$\partial z_{2}(\overline{Q}) \subset \partial z_{1}(\overline{Q})$ since $\partial z_{2}(Q)\subset \partial z_{1}(Q)$.

In the following, let $T_{i}$ be a supporting hyperplane of the graph of $z_{i}$ at $(\bx_{0}, z_{i}(\bx_{0}))$ 
 and the equations of $T_{i}$ be
\begin{align*}
& \quad z = z_{i}(\bx_0)+\bp_{i}\cdot (\bx - \bx_{0}),
\end{align*}
for $i=1,2$, $\bp_{1}\in \partial z_{1}(\bx_{0})$ and $\bp_{2}=\nabla z_{2}(\bx_{0})$. 
For any $0<\lambda<1$, let $T_{\lambda}$ be a hyperplane given by the equation:
$
z = z_{1}(\bx_{0}) + \bp_{\lambda}\cdot (\bx - \bx_{0}), 
$
where $\bp_{\lambda}:=(1-\lambda)\bp_{2}+\lambda\bp_{1}$. 
Obviously, $ (\bx_{0},z_{1}(\bx_{0})) = (\bx_{0}, z_{2}(\bx_{0}))\in T_{\lambda}, $ and 
$ z_{2}(\bx) \geq z_{2}(\bx_{0}) + \bp_{\lambda}\cdot (\bx - \bx_{0}),$  $ \forall \bx \in \overline{Q}$  and  any $  0<\lambda<1.
$
Since $z_{2}$ is differentiable at $\bx_{0}$, then we claim that $$\bp_{\lambda} \notin \partial z_{2}(\bx_{0}), \ \forall 0<\lambda < 1.$$ In fact, if not, then $\exists \lambda_0\in (0, 1)$ such that $\bp_{\lambda_0} = \nabla z_{2}(\bx_{0}),$ which implies that $\bp_1=\bp_2$. Then we arrive at a contradiction. 

Since $\displaystyle\bp_{2} = \lim_{\lambda\rightarrow 0+}\bp_{\lambda}$ and $\bp_{\lambda} \notin \partial z_{2}(\bx_{0}), \ \forall 0<\lambda < 1$, then $\bp_{2}$ belongs to 
the boundary of $\partial z_{2}(\overline{Q})$. 
Applying the argument before for $\bp_{T^{\prime}}$ to the fact: $\bp_{2} = \nabla z_{2}(\bx_{0})\notin \partial z_{1}(\bx_{0})$, 
we get $\bp_{2}\in \text{Int}(\partial z_{1}(Q))$. 
By \cite[Property B in Section $9.4$]{Bakelman94}, $\partial z_{2}(\overline{Q})$ is compact in $\mathbb{R}^{d}$. 
Then we have 
\begin{align*}
\int_{\partial z_{1}(\overline{Q})}R(\bp) d\bp > \int_{\partial z_{2}(\overline{Q})}R(\bp) d\bp.
\end{align*}
\end{proof}

Finally we go to the proof of Theorem \ref{thm_comparison}:

{\bf Proof of Theorem~\ref{thm_comparison}}.
We prove it by contradiction. Assume $\{\bx\in\Omega: z_{1}(\bx)<z_{2}(\bx)\}\neq \emptyset$ and  $Q$ is a
connected component of $\{\bx\in\Omega: z_{1}(\bx)<z_{2}(\bx)\}$. We define
$
\epsilon_{0}  := \sup_{\bx\in Q}( z_{2}(\bx) - z_{1}(\bx) ),$  $\ z_{2}^{(1)} \  := \ z_{2}(\bx) - \epsilon_{0}/2$ and
$Q^{(1)}  := \{\bx\in Q: z_{1}(\bx) < z_{2}^{(1)}(\bx) \}.
$
Obviously, $\epsilon_0 >0$ and $Q^{(1)}\neq \emptyset$. Without the loss of generality, we assume that $Q^{(1)}$ is connected. 
From the assumptions on $z_1$ and $z_2$, 
we know that $\overline{Q^{(1)}} \subset \Omega$. Then by Lemma~\ref{lemma_comparision_aux2},  
for any $\bx\in \partial \overline{Q^{(1)}}$ such that $z_{2}^{(1)}$ is differentiable at $\bx$,  it holds:
$\nabla z_{2}(\bx) = \nabla z_{2}^{(1)}(\bx)\in \partial z_{1}(\bx). $

In the following, we claim,  any $\bx_{0}\in Q^{(1)}$ such that $z_{2}$ is differentiable at $\bx_{0}$,  that 
\begin{align}
\label{claim_thm_comparison}
\nabla z_{2}(\bx_{0}) \in \partial z_{1}(\bx_{0}). 
\end{align}
If (\ref{claim_thm_comparison}) is not true, we define 
$
\epsilon_{1}  := z_{2}^{(1)}(\bx_{0}) - z_{1}(\bx_{0}),\
z_{2}^{2}  := z_{2}^{(1)} - \epsilon_{1},\
Q^{(2)}  := \{ \bx\in Q^{(1)}: z_{1}(\bx) < z_{2}^{(2)}(\bx) \}.
$
Obviously, $z_{2}^{(1)}(\bx_{0}) > z_{2}^{(2)}(\bx_{0}) = z_{1}(\bx_{0})$ and we infer

\begin{itemize}

\item[(1)] $Q^{(2)} \neq \emptyset$. In fact, if not, then $z_{1}\geq z_{2}^{(2)}$ in $Q^{(1)}$. 
Since $z_{1}(\bx_{0}) = z_{2}^{(2)}(\bx_{0})$, then 
$
\nabla z_{2}(\bx_{0}) = \nabla z_{2}^{(2)}(\bx_{0})\in \partial z_{1}(\bx_{0}),
$
which contradicts with our assumption.

\item[(2)] $\bx_{0}\in \overline{Q^{(2)}}$.  If not, then $\exists r_{1}>0$ such that 
$
\overline{B_{r_{1}}(\bx_{0})} \subset Q^{(1)}\text{  and  } B_{r_{1}}(\bx_{0})\cap Q^{(2)} = \emptyset. 
$
Hence $z_{1} \geq z_{2}^{(2)}$ on $\overline{B_{r_{1}}(\bx_{0})}$. Since $z_{1}(\bx_{0}) = z_{2}^{2}(\bx_{0})$,  then it holds:
$\nabla z_{2}(\bx_{0})\in \partial z_{1}(\bx_{0})$, which is a contradiction. 

\item[(3)] $\bx_{0}\in \partial \overline{Q^{(2)}}$. 
If not, then by (2), we know that $\bx_{0}\in \text{Int}(\overline{Q^{(2)}})$. 
That is, $\exists r_{2}>0$ such that 
$
B_{r_{2}}(\bx_{0})\subset \text{Int}(\overline{Q^{(2)}})\text{  and  } \overline{B_{r_{2}}(\bx_{0})} \subset Q^{(1)}.
$
Hence $z_{1} \leq z_{2}^{(2)}$ in $B_{r_{2}}(\bx_{0})$.  Since $z_{1}(\bx_{0}) = z_{2}^{2}(\bx_{0})$, we get
$\partial z_{1}(\bx_{0}) = \{ \nabla z_{2}(\bx_{0}) \}$, which is contradiction. 
\end{itemize}
Since $\bx_{0}\in \partial \overline{Q^{(2)}}$,  there exists $r_{3}>0$ such that 
$
\overline{B_{r_{3}}(\bx_{0})} \subset Q^{(1)}$ and  $z_{1} \geq z_{2}^{(2)} \text{ in } B_{r_{3}}(\bx_{0})\backslash Q^{(2)}.
$
By Lemma~\ref{lemma_comparision_aux2}, it holds:
\begin{align*}
\int_{\partial z_{1}(\overline{Q^{(2)}})}R(\bp) d\bp > \int_{\partial z_{2}(\overline{Q^{(2)}})}R(\bp) d\bp,
\end{align*}
which arrives at a contradiction.

 Since $z_{1}, z_{2}, z_{2}^{(1)}, z_{2}^{(2)}\in W^{+}(\Omega)$ , then they are all differentiable almost everywhere in $\Omega$. By  (2.4),  we know that
$
\nabla (z_{1} - z_{2}^{(1)})(\bx) = 0\quad  \text{for a. e.} \ \bx \in Q^{(1)},$  which, together with \cite[Corollary~$2.1.9$]{Ziemer89} and the fact: $z_{1} = z_{2}^{(1)}$ on $\partial Q^{(1)}$ to (2.5) , implies that  
$
z_{1} = z_{2}^{(1)}\text{ in } \overline{Q^{(1)}}.
$
This is a contraction with our assumption.    \hspace{280pt}$\Box$

\vspace{35pt}

\section{The border of a convex function}
In this part, the boundary behaviour of convex functions would be considered. Firstly, we give the definition of border of convex functions on convex domains, which  was firstly introduced by  Bakelman\cite[Section~$10.4$]{Bakelman94}.

\begin{definition}
\label{def_border}
(The border of a convex function) 
For any $v\in W^{+}(\Omega)$, we define the border of $v$ to be a function on $\partial\Omega$ by  
\begin{align*}
b_{v}(\bx_{0}) = \liminf_{\bx\rightarrow \bx_{0}}v(\bx),\quad \forall \bx_{0}\in \partial\Omega.
\end{align*}
\end{definition}
The following Lemma~\ref{lemma_continuous_border} shows that if $b_{v}\in C^{0}(\partial\Omega)$, 
then $v$ can be extended continuously to $\overline{\Omega}$ such that $v|_{\partial\Omega} = b_{v}$.  

\begin{lemma}
\label{lemma_continuous_border}
Let $v\in W^{+}(\Omega)$ and $b_{v}$ be the border of $v$ as in Definition~\ref{def_border}. 
For any $\bx_{0}\in\partial\Omega$, if $b_{v}$ is continuous at $\bx_{0}$, then we have
\begin{align*}
b_{v}(\bx_{0}) = \lim_{\bx\rightarrow \bx_{0}} v(\bx).
\end{align*}
\end{lemma}

\begin{proof}
We choose an orthogonal coordinate of $\mathbb{R}^{d}$ such that in an open neighborhood of $\bx^{0}$, 
$\partial\Omega$ can be represented as 
\begin{align*}
(y^{1},\cdots, y^{d-1}, z(y^{1},\cdots, y^{d-1})) \text{  satisfying  } z(y^{1},\cdots, z^{d-1})\geq 0.
\end{align*} 
For simplicity,  $\bx^{0}$ is taken to be the origin and for any $\bx\in \mathbb{R}^{d}$, we denote by $(x^{1},\cdots, x^{d})$ its coordinate. 
We define  
$$\displaystyle
\tilde{b}_{v}(\bx) := \limsup_{\by\rightarrow \bx} v(\by), \forall \bx\in\partial\Omega.
$$
To prove Lemma 3.2, it is sufficient to show that  $\tilde{b}_{v}(\bx_{0}) = b_{v}(\bx_{0})$ and 
 we prove it by contradiction. Assume $\epsilon:= \tilde{b}_{v}(\bx_{0}) - b_{v}(\bx_{0})>0$.  
By the continuity of $b_{v}$  at $\bx_{0}$, there is $\delta_{1}>0$ such that the following holds:
\begin{align*}
\vert b_{v}(y^{1},\cdots,y^{d-1}, z(y^{1},\cdots, y^{d-1})) - b_{v}(\bx_{0})\vert < \epsilon/3,
\quad \forall (y^{1},\cdots, y^{d-1})\in B_{2\delta_{1}}(\boldsymbol{0})\subset \mathbb{R}^{d-1}. 
\end{align*} 
We define 
\begin{align*}
S&:=\{\bx\in \Omega: \  \vert (x^{1},\cdots, x^{d-1}) \vert< 2\delta_{1}\text{ and } \exists 
(y^{1},\cdots, y^{d-1})\in \overline{B_{\delta_{1}}(\boldsymbol{0})}\subset \mathbb{R}^{d-1}\\
 & \text{ such that it holds: } \vert b_{v}(y^{1},\cdots,y^{d-1}, z(y^{1},\cdots, y^{d-1})) - v(\bx)\vert
 < \epsilon/3\}.
\end{align*}
Then we claim that for any $\delta_{2}>0$ and any $(y^{1},\cdots, y^{d-1})\in \overline{B_{\delta_{1}}(\boldsymbol{0})}
\subset \mathbb{R}^{d-1}$, there exists $\bx \in S$ such that 
\begin{align}
\label{S_dense}
\vert (y^{1},\cdots, y^{d-1}, z(y^{1},\cdots, y^{d-1})) - \bx\vert < \delta_{2}.
\end{align}
In fact, if (\ref{S_dense}) is not true, then there exist $0 < \bar{\delta}_{2} \leq \delta_{1}$ and 
$(\bar{y}^{1},\cdots, \bar{y}^{d-1})\in \overline{B_{\delta_{1}}(\boldsymbol{0})}\subset \mathbb{R}^{d-1}$, 
such that for any $\bx\in S$, we infer
\begin{align*}
\vert (\bar{y}^{1},\cdots, \bar{y}^{d-1}, z(\bar{y}^{1},\cdots, \bar{y}^{d-1})) - \bx\vert \geq \bar{\delta}_{2}>0.
\end{align*}
By the definition of $S$, for any $\bx \in \Omega$ 
with  $\vert (\bar{y}^{1},\cdots, \bar{y}^{d-1}, z(\bar{y}^{1},\cdots, \bar{y}^{d-1})) - \bx\vert < \bar{\delta}_{2}$,  we know that $\bx \notin S$ and 
\begin{align*}
\vert b_{v}(\bar{y}^{1},\cdots,\bar{y}^{d-1}, z(\bar{y}^{1},\cdots, \bar{y}^{d-1})) - v(\bx)\vert \geq \epsilon/3 > 0,
\end{align*}
which contradicts with the definition of $b_{v}$. Thus the claim holds true.

It is easy to see that there exist $d$ points $\{(y_{i}^{1},\cdots, y_{i}^{d-1})\}_{i=1}^{d}$  such that $\vert (y_{i}^{1},\cdots, y_{i}^{d-1})\vert = \delta_{1} $  for all $1\leq i\leq d$  and
\begin{align*}
\sum_{i=1}^{d}d^{-1}(y_{i}^{1},\cdots, y_{i}^{d-1}) = \boldsymbol{0} \in \mathbb{R}^{d-1}. 
\end{align*}
By (\ref{S_dense}),  there exist a constant $0<\sigma<1$ and 
$\{\bx_{i}\}_{i=1}^d$ in $S$, such that 
\begin{align}
\label{general_points}
\overline{B_{\sigma \delta_{1}}(\boldsymbol{0})}\subset \mathbb{R}^{d-1} \text{ is contained in the 
convex hull of } \{(x_{i}^{1},\cdots, x_{i}^{d-1})\}_{i=1}^{d} \text{ in } \mathbb{R}^{d-1}. 
\end{align}
Let $T$ be the hyperplane in $\mathbb{R}^{d}$ passing through $\{\bx_{i}\}_{i=1}^{d}$. Then the equation of $T$ is 
$
x^{d} = a^{1}x^{1}+\cdots + a^{d-1}x^{d-1} + c.
$
Since $x_{i}^{d}>0$ for any $1\leq i\leq d$, then $c>0$. Thus from  the definition of $\tilde{b}_{v}$,  there is $\bar{\bx}\in \Omega$ such that 
\begin{subequations}
\begin{align}
\label{vertex_point1}
& (\bar{x}^{1},\cdots, \bar{x}^{d-1})\in B_{\frac{1}{2}\sigma\delta_{1}}(\boldsymbol{0})\subset \mathbb{R}^{d-1},\\
\label{vertex_point2}
& \vert \tilde{b}_{v}(\bx_{0}) - v(\bar{\bx}) \vert < \epsilon/3,\\
\label{vertex_point3}
& \bar{x}^{d} < a^{1}\bar{x}^{1} + \cdots + a^{d-1}\bar{x}^{d-1} + c. 
\end{align}
\end{subequations}
Here (\ref{vertex_point3}) holds true since  $\bar{\bx}$ can  be chosen  as close to $\bx_{0}=\boldsymbol{0}$ as we need. 
By (\ref{general_points}, \ref{vertex_point1}), there are $0< \mu_{i}<1$ for any $1\leq i \leq d$ such that $\mu_{1}+\cdots + \mu_{d} = 1$ and
\begin{align*}
& (\bar{x}^{1},\cdots, \bar{x}^{d-1}) = \sum_{i=1}^{d}\mu_{i} (x_{i}^{1},\cdots, x_{i}^{d-1}), \ \  \sum_{i=1}^{d}\mu_{i} (x_{i}^{1},\cdots, x_{i}^{d-1}, x_{i}^{d})\in T.
\end{align*}
 
By the fact that $z(\bar{x}^{1},\cdots, \bar{x}^{d-1})< \bar{x}^{d}$ and (\ref{vertex_point3}), 
there is $0< \lambda < 1$ such that 
$$\bar{x}^{d} = (1-\lambda)z(\bar{x}^{1},\cdots, \bar{x}^{d-1}) + \lambda \sum_{i=1}^{d}\mu_{i} x_{i}^{d}, $$
and $(\bar{x}^{1},\cdots, \bar{x}^{d-1}, z(\bar{x}^{1},\cdots, \bar{x}^{d-1})), \bx_{1},\cdots, \bx_{d} 
\text{ are in a general location in } \mathbb{R}^{d}. 
$
Hence we get
\begin{align*}
\bar{\bx} = (1-\lambda)(\bar{x}^{1},\cdots, \bar{x}^{d-1}, z(\bar{x}^{1},\cdots, \bar{x}^{d-1})) 
+ \lambda \sum_{i=1}^{d}\mu_{i}(\bar{x}_{i}^{1},\cdots, x_{i}^{d-1}, x_{i}^{d}),
\end{align*}
and $\bar{\bx}$ is contained in the interior of the convex hull of $(\bar{x}^{1},\cdots, \bar{x}^{d-1}, 
z(\bar{x}^{1},\cdots, \bar{x}^{d-1}))\cup \{\bx_{i}\}_{i=1}^{d}$. Due to (\ref{S_dense}), we can take 
$\bx_{d+1}\in S$ close enough to $(\bar{x}^{1},\cdots, \bar{x}^{d-1}, z(\bar{x}^{1},\cdots, \bar{x}^{d-1}))$, 
such that the point $\bar\bx$ is contained in the convex hull of $\{\bx_{i}\}_{i=1}^{d+1}$.  Then we can infer
\begin{align*}
v(\bar{\bx})& \leq \max_{1\leq i\leq d+1} v(\bar{x}_{i}). 
\end{align*}
By the construction of the set $S$, $v(\bx)< b_{v}(\bx_{0})+{2}\epsilon/3$ for any $\bx\in S$. Thus we have
\begin{align*}
v(\bar{\bx}) < b_{v}(\bx_{0})+{2}\epsilon/3 = \tilde{b}_{v}(\bx_{0}) - \epsilon/3.
\end{align*}
This contradicts with (\ref{vertex_point2}). 
\end{proof}
\vspace{35pt}

\section{Convergence of a sequence of convex functions}
Throughout this section, we denote by $\{\Omega_{n}\}_{n=1}^{+\infty}$ a sequence of open 
convex subdomains of $\Omega$, and $\{v_{n}\}_{n=1}^{+\infty}$ a sequence of convex functions with
\begin{align}
\label{assmp_function_sequence}
v_{n} \in W^{+}(\Omega_{n}),\quad \forall n\in \mathbb{N}.
\end{align}
Furthermore, we assume that for any $\delta>0$, there is $N=N(\delta)\in\mathbb{N}$,  
\begin{align}
\label{assmp_subdomains}
\overline{\Omega_{\delta}} \subset \Omega_{n} \subset \Omega, \quad \forall n\geq N.
\end{align}
This section would consist of the following two parts:

\subsection{Convergence of a sequence of convex functions inside domain}\

\vspace{15pt}

The main result of this subsection is Theorem~\ref{thm_convex_sequence_converge_domain}. 
\begin{theorem}
\label{thm_convex_sequence_converge_domain}
We assume that (\ref{assmp_function_sequence}, \ref{assmp_subdomains}) hold and there is $M< +\infty$ such that  it holds:
\begin{align}
\label{assmp_convex_functions}
\Vert v_{n}\Vert_{L^{\infty}(\Omega_{n})} \leq M, \quad \forall n\in\mathbb{N}. 
\end{align}
Then there is a subsequence $\{v_{n_{k}}\}_{k=1}^{+\infty}$ of $\{v_n\}^{+\infty}_{n=1}$ and a function $v_{0}\in W^{+}(\Omega)$, 
such that for any $\delta>0$, 
\begin{align*}
\Vert v_{n_{k}} -  v_{0}\Vert_{L^{\infty}(\overline{\Omega_{\delta}})}\ {\longrightarrow}\ 0 \hspace{20pt} \text{as} \ k\rightarrow +\infty.
\end{align*}
Moreover, if we define the set functions $\nu_{k}$ and $\nu_{0}$ by
\begin{align*}
\nu_{k}(e) & := \int_{\partial v_{n_{k}}(e)} R(\bp) d\bp,\quad \forall \text{ Borel set } e\subset \Omega_{n_{k}},\\
\nu_{0}(e) & := \int_{\partial v_{0}(e)} R(\bp) d\bp,\quad \forall \text{ Borel set } e\subset \Omega.
\end{align*}
Then, $\nu_{0}$ is a measure in $\Omega$, and $\nu_{k}$ is a measure in $\Omega_{n_{k}}$ 
for any $n\in \mathbb{N}$. Furthermore, it holds:
$
\nu_{k} \rightharpoonup \nu_{0}\text{ weakly, as }k\rightarrow +\infty,
$
i.e. for any $f\in C_{c}(\Omega)$,  it holds:
\begin{align*}
\int_{\Omega_{n_{k}}}f d\nu_{k} \rightarrow \int_{\Omega}f d\nu_{0},\text{ as } k\rightarrow +\infty.
\end{align*}
\end{theorem}

The proof of Theorem~\ref{thm_convex_sequence_converge_domain} comes from Lemma~\ref{lemma_convex_sequence} 
and Lemma~\ref{lemma_measure_weak_converge} immediately.

\begin{lemma}
\label{lemma_convex_sequence}
We assume that (\ref{assmp_function_sequence}), (\ref{assmp_subdomains}) and (\ref{assmp_convex_functions}) hold. 
Then there is a subsequence $\{v_{n_{k}}\}_{k=1}^{+\infty}$ of $\{v_n\}^{+\infty}_{n=1}$ and a function $v_{0}\in W^{+}(\Omega)$, 
such that for any $\delta>0$, 
\begin{align*}
\Vert v_{n_{k}} -  v_{0}\Vert_{L^{\infty}(\overline{\Omega_{\delta}})}\ {\longrightarrow}\ 0 \ \hspace{15pt}\text{as}\  k\rightarrow +\infty.
\end{align*}
\end{lemma}

\begin{proof}
By (\ref{assmp_subdomains}), we know that there exists some $N=N(\delta)\in \mathbb{N}$ such that 
$
\overline{\Omega_{\delta}} \subset \Omega_{n}$   and   $ \text{dist}(\partial \Omega_{n}, \Omega_{\delta}) 
\geq {\delta}/{2}$ for all  $ n\geq N.$
By (\ref{assmp_convex_functions}), it is easy to see that 
\begin{align*}
\sup_{\bx, \by\in \overline{\Omega_{\delta}}} \vert v_{n}(\bx) - v_{n}(\by)\vert\leq \varrho_{\delta}\cdot {\delta}/{2}, 
\quad \forall n\geq N.
\end{align*}
Here $
\varrho_{\delta}:= {4M}/{\delta}.
$
By the convexity of $\{v_{n}\}_{n=1}^{+\infty}$, we infer that 
$
\partial v_{n}(\overline{\Omega_{\delta}}) \subset \overline{B_{\varrho_{\delta}}(\boldsymbol{0})}
 \subset \mathbb{R}^{d}$ for all  $n\geq N.
$
Therefore for any $\bp\in \cap_{n\geq N}\partial v_{n}(\overline{\Omega_{\delta}})$, we have that $ \vert \bp\vert \leq \varrho_{\delta} $ and
\begin{align*}
 v_{n}(\bx) - v_{n}(\by) \geq \bp\cdot (\bx - \by) 
\geq -\varrho_{\delta}\cdot \vert \bx - \by \vert,\quad \forall \bx,\by \in \overline{\Omega_{\delta}}.
\end{align*}
This statement implies the equicontinuity of $\{v_{n}\}_{n\geq N}$ on $\overline{\Omega_{\delta}}$. 

Therefore, by Ascoli-Arzel{\`{a}} Theorem, there exists a function $v_{0}\in W^{+}(\Omega)$ and 
a subsequence $\{v_{n_{k}}\}_{k=1}^{+\infty}$ of $\{v_n\}^{+\infty}_{n=1}$, such that for any $\delta>0$
\begin{align*}
\lim_{k\rightarrow +\infty} \Vert v_{n_{k}} - v_{0}\Vert_{L^{\infty}(\overline{\Omega_{\delta}})}=0.
\end{align*}
\end{proof}

\begin{lemma}
\label{lemma_measure_weak_converge}
Let (\ref{assmp_function_sequence}, \ref{assmp_subdomains}) hold. We assume that there is 
a function $v_{0}\in W^{+}(\Omega)$, such that for any $\delta>0$,
\begin{align}
\label{uniform_conv_compactly}
\lim_{n\rightarrow +\infty} \Vert v_{n} - v_{0}\Vert_{L^{\infty}(\overline{\Omega_{\delta}})}=0.
\end{align} 
We define the set functions $\nu_{n}$ and $\nu_{0}$ by
\begin{align*}
\nu_{n}(e) & := \int_{\partial v_{n}(e)} R(\bp) d\bp,\quad \forall \text{ Borel set } e\subset \Omega_{n},\\
\nu_{0}(e) & := \int_{\partial v_{0}(e)} R(\bp) d\bp,\quad \forall \text{ Borel set } e\subset \Omega.
\end{align*}
Then, $\nu_{0}$ is a measure in $\Omega$, and $\nu_{n}$ is a measure in $\Omega_{n}$ 
for any $n\in \mathbb{N}$. Furthermore, for any $f\in C_{c}^{0}(\Omega)$, 
\begin{align*}\int_{\Omega_{n}}f d\nu_{n} \rightarrow \int_{\Omega}f d\nu_{0},\text{ as } n\rightarrow +\infty.
\end{align*}
\end{lemma}

\begin{proof}
Since $R>0$ and $R\in L_{loc}^{1}(\mathbb{R}^{d})$, thus $\nu_{0}$ is a measure in $\Omega$, 
and $\nu_{n}$ is a measure in $\Omega_{n}$ for any $n\in\mathbb{N}$ (see \cite[Theorem~$1.1.13$]{Gutierrez01}). 

By (\ref{assmp_subdomains}), for any compact set $F\subset \Omega$ and any open set $Q$ with $\overline{Q}\subset \Omega$, 
we have that $F\subset \Omega_{n}$ {  and  } $\overline{Q} \subset \Omega_{n}$ for any $n\in\mathbb{N}$ large enough.
By (\ref{uniform_conv_compactly}) and \cite[Lemma~$1.2.2$]{Gutierrez01}, there hold:
$
\limsup_{n\rightarrow +\infty}\partial v_{n}(F) \subset \partial v_{0}(F)$ {  and  } 
$\liminf_{n\rightarrow +\infty} \partial v_{n}(Q)\supset\partial v_{0}(K)$ { for any compact set } $K\subset Q.
$
Then by Fatou Lemma, we obtain
\begin{align}
\label{measure_weak_converge1}
\limsup_{n\rightarrow +\infty} \nu_{n}(F) \leq \nu_{0}(F) \text{  and  }
\liminf_{n\rightarrow +\infty} \nu_{n}(Q) \geq \nu_{0}(Q),
\end{align}
which implies that 
\begin{align}
\label{measure_weak_converge2}
\lim_{n\rightarrow +\infty} \nu_{n}(B) = \nu_{0} (B). 
\end{align}for any Borel set $B\subset \Omega$ with $\overline{B}\subset \Omega$ and 
$\nu_{0}(\partial B)=0$.
Now we choose $f\in C_{c}^{0}(\Omega)$ with $f\geq 0$. (in fact, we can write $f= f^{+} - f^{-}$). 
By (\ref{assmp_subdomains}), we have that for $n\in \mathbb{N}$ large enough, 
\begin{align*}
\int_{\Omega_{n}} f d\nu_{n} = \int_{0}^{+\infty}
\nu_{n}(\{\bx\in\Omega_{n}: f(\bx)> t\})dt 
= \int_{0}^{+\infty} \nu_{n}(\{\bx\in\Omega: f(\bx)> t\})dt.
\end{align*}
Let $B_{t}:=\{\bx\in\Omega: f(\bx)> t\},\forall t>0$. 
Since $f\in C_{c}^{0}(\Omega)$, then      $B_{t}$ { is Borel  and } $\partial B_{t} \subset A_{t} := \{\bx\in \Omega: f(\bx) =t \}$ for any $t>0$.
Furthermore, if $n\in \mathbb{N}$ large enough, $B_{t} \subset \text{Supp}(f) \subset \Omega_{n}.$
By foliations of Borel sets (see \cite[Proposition $2.16$]{Maggi2012}), we know that 
$\nu_{0}(A_{t})>0$ for at most countably many $t\in (0,+\infty)$. 
Hence, $\exists J \subset (0, +\infty)$ with $\vert J\vert>0$ and $\vert (0,+\infty)\backslash J \vert = 0$ 
such that $\nu_{0}(A_{t})=0, \forall t\in J$. This implies that $\nu_{0}(\partial B_{t}) = 0,\forall t\in J$. 
From (\ref{measure_weak_converge2}), we obtain 
\begin{align}
\label{measure_weak_converge3}
\lim_{n\rightarrow +\infty} \nu_{n}(B_{t}) = \nu_{0}(B_{t}),\quad \forall t\in J.
\end{align}
Moreover by (\ref{measure_weak_converge1}), there is a positive constant $C_{0}$ such that for any $n\in\mathbb{N}$
\begin{align}
\label{measure_weak_converge4}
\nu_{n}(B_{t}) &\leq  C_{0}\big( 1 + \nu_{0}(\text{Supp}(f)) \big)\cdot \chi_{[0, \overline{f}]}(t)
\end{align}where $\overline{f}:= \displaystyle\sup_{\bx\in \Omega}f(\bx)$ and $\displaystyle\chi_{[0, \overline{f}]}$ denotes the characteristic function on $[0, \overline{f}]$.
By dominated convergence Theorem, we infer, from (\ref{measure_weak_converge3}) and (\ref{measure_weak_converge4}), that 
\begin{align*}
\int_{\Omega_{n}} f d\nu_{n} = \int_{0}^{+\infty} \nu_{n}(B_{t}) dt 
{\longrightarrow} \int_{0}^{+\infty}\nu_{0}(B_{t}) dt = \int_{\Omega} f d\nu_{0}, \ \text{as}\ {n\rightarrow +\infty}.
\end{align*}
\end{proof}
\vspace{20pt}

\subsection{Convergence of a sequence of borders of convex functions}\

\vspace{10pt}

In this subsection, we study the convergence property of border of a sequence of convex functions. 
The main results in this part are Lemma~\ref{lemma_border_upper_bound} and Theorem~\ref{thm_border_converge}.  

Before giving  main results, we firstly need to  introduce some important notations (see also Bakelman[3]). Let $\ba_{0}$ be any point of $\partial\Omega$. Then there is a supporting $(d-1)$-plane $\alpha$ of $\partial\Omega$ 
passing through $\ba_{0}$, an open $d$-ball $U_{\rho}(\ba_{0})$ with the center $\ba_{0}$, and the radius $\rho>0$ 
such that the convex $(d-1)$-surface 
\begin{align*}
\Gamma_{\rho}(\ba_{0}) := \partial\Omega \cap U_{\rho}(\ba_{0})
\end{align*}
has the one-to-one orthogonal projection $\Pi_{\alpha}:\Gamma_{\rho}(\ba_{0})\rightarrow \alpha$. 
Moreover, the unit normal of $\alpha$ in the direction of the halfspace of $\mathbb{R}^{d}$, where 
$\overline{\Omega}$ stays, passes through interior points of $\Omega$. 
Let $x^{1},\cdots, x^{d-1}, x^{d},z$ be the Cartesian coordinates in $\mathbb{R}^{d+1}$ with the following properties:
\begin{itemize}
\item
$\ba_{0}$ is the origin. 

\item
The axes $x^{1},\cdots,x^{d-1}$ stay in the plane $\alpha$.

\item
The axis $x^{d}$ is directed along the interior normal of $\partial\Omega$ at the point $\ba_{0}$.

\item
The axis $z$ is orthogonal to $\mathbb{R}^{d}$. 
\end{itemize}
Clearly, the convex $(d-1)$-surface $\Gamma_{\rho}(\ba_{0})$ is the graph of 
$g(x^{1},\cdots, x^{d-1})\in W^{+}(\Pi_{\alpha}(\Gamma_{\rho}(\ba_{0})))$. Obviously, 
$
g(0,\cdots,0)=0 \text{ and  } g(x^{1},\cdots, x^{d-1})\geq 0 
\text{ for all points of the set } \Pi_{\alpha}(\Gamma_{\rho}(\ba_{0})).
$

\begin{definition}
\label{def_local_para}
(Local parabolic support)
We shall say that $\partial\Omega$ has a local parabolic support of order $\tau \geq 0$ at the point $\ba_{0}$ 
if there are positive numbers $\rho_{0}$ and $\eta(\ba_{0})$ such that 
\begin{align*}
g(x^{1},\cdots,x^{d-1})\geq \eta(\ba_{0}) \big(\vert x^{1}\vert^{2}+\cdots + \vert x^{d-1}\vert^{2} \big)^{\frac{\tau+2}{2}}, 
\quad \forall (x^{1},\cdots,x^{d-1})\in \Pi_{\alpha}(\Gamma_{\rho}(\ba_{0})).
\end{align*}
\end{definition}

\begin{definition}
\label{def_para}
(Boundary having a parabolic support)
We shall say that $\partial\Omega$ has a parabolic support of order not smaller than a constant $0\leq\tau <+\infty$, 
if the local parabolic support of $\partial\Omega$ has order not smaller than $\tau$ at all points $\ba\in\partial\Omega$.
\end{definition}

\begin{definition}
\label{def_topological_limit}
(Topological limit of sets\cite[Section~$3.4$]{Bakelman94})
Let $\{E_{n}\}_{n=1}^{+\infty}$ be a sequence of subsets in $\mathbb{R}^{d}$ and we denote by $\displaystyle {\overline{\lim}^T_{n\rightarrow +\infty}}E_{n}$ the superior topological limit of 
$\{E_{n}\}_{n=1}^{+\infty}$,  which is defined as 
\begin{align*}
\bx\in {\overline{\lim}}^T_{n\rightarrow +\infty} E_{n}\Leftrightarrow \exists\text{ a subsequence} \{n_{k}\}_{k=1}^{\infty} 
\text{ and }\bx_{n_{k}}\in E_{n_{k}} \text{ such that}\lim_{k\rightarrow +\infty}\bx_{n_{k}}  = \bx.
\end{align*} We also denote by $\displaystyle {\underline{\lim}}^T_{n\rightarrow +\infty}E_{n}$ the inferior topological limit of 
$\{E_{n}\}_{n=1}^{+\infty}$,  which is defined as 
\begin{align*}
\bx\in \displaystyle {\underline{\lim}}^T_{n\rightarrow +\infty}E_{n}\Leftrightarrow \exists\bx_{n}\in E_{n} 
\text{ such that}\lim_{n\rightarrow +\infty}\bx_{n}  = \bx.
\end{align*}
If 
$
\displaystyle {\overline{\lim}}^T_{n\rightarrow +\infty}E_{n}  = \displaystyle {\underline{\lim}}^T_{n\rightarrow +\infty}E_{n},
$
then we say $\{E\}_{n=1}^{+\infty}$ has a topological limit, written as $\lim^T_{n\rightarrow +\infty}E_{n}$, which is equal to ${\overline{\lim}}^T_{n\rightarrow +\infty}E_{n} $  or $ \displaystyle {\underline{\lim}}^T_{n\rightarrow +\infty}E_{n}$.
\end{definition}\

\begin{assumption}
\label{assmp_boundary_para}
$\partial\Omega$ has a parabolic support (see Definition~\ref{def_para}) of order not smaller than a constant $0\leq \tau < +\infty$.
\end{assumption}
\begin{remark} if Assumption~\ref{assmp_boundary_para} holds for $\partial \Omega$, then the domain $\Omega$ is strictly convex.
\end{remark}

\begin{assumption}
\label{assmp_R}
$R\in L_{loc}^{1}(\mathbb{R}^{d})$  and the following holds:
\begin{align*}
R(\bp) \geq C_{0}\vert \bp\vert^{-2k},\quad \forall \bp\in \mathbb{R}^{d} 
\text{ with } \vert \bp \vert \geq r_{0}>0.
\end{align*} 
Here, $k\geq 0$, $C_{0}>0$ and $r_{0}>0$ are some constants.
\end{assumption}

\begin{assumption}
\label{assmp_function_decay}
Let $\{\Omega_{n}\}_{n=1}^{+\infty}$ be a sequence of open convex subdomains  of $\Omega$ 
satisfying (\ref{assmp_subdomains}), and $\{v_{n}\}_{n=1}^{+\infty}$ a sequence of convex functions 
satisfying (\ref{assmp_function_sequence}).
We assume that the following conditions are fulfilled:
\begin{itemize}
\item[(a)] 
There is a function $v_{0}\in W^{+}(\Omega)$ such that 
$
\displaystyle \lim_{n\rightarrow +\infty}v_{n}(\bx)\ =\   v_{0}(\bx),\quad \forall \bx \in \Omega.
$

\item[(b)]
There exist  two uniform constants $C_{1}>0$ and $\lambda \geq 0$ such that  for any $\bx_{0}\in \partial\Omega$, there exists an open $d$-ball $U_{\rho}(\bx_{0})$ such that 
\begin{align*}
\liminf_{n\rightarrow +\infty}\int_{\partial v_{n}(e \cap \Omega_{n})}R(\bp)d\bp
\leq C_{1}\big( \sup_{\bx \in e}  \text{dist}(\bx, \partial\Omega)\big)^{\lambda}\vert e\vert,
\quad \forall \text{ Borel set } e\subset U_{\rho}(\bx_{0})\cap \Omega.
\end{align*}
\end{itemize}
\end{assumption}

\begin{assumption}
\label{assmp_border}
Let $\{\Omega_{n}\}_{n=1}^{+\infty}$ be a sequence of open convex subdomains  of $\Omega$ 
satisfying (\ref{assmp_subdomains}), and $\{v_{n}\}_{n=1}^{+\infty}$ be a sequence of convex functions 
satisfying (\ref{assmp_function_sequence}). 
Let $b_{n}$ be the border of $v_{n}$ for any $n\in\mathbb{N}$. 
Let $\bS_{n}$ be the graphs of $b_{n}$ for any $n\in\mathbb{N}$. We assume that 
\begin{itemize}
\item[(a)]
$b_{n}\in C^{0}(\partial\Omega_{n})$ for any $n\in \mathbb{N}$.

\item[(b)]
There is $\tilde{b}\in C^{0}(\partial \Omega)$ such that 
$\lim^T_{n\rightarrow +\infty}\bS_{n} = \tilde{\bS}$,  
where $\tilde{\bS}$ is the graph of $\tilde{b}$. 
\end{itemize}
\end{assumption}

The following Lemma \ref{lemma_border_upper_bound} gives some important convergent property for the borders of a sequence of convex functions:
\begin{lemma}
\label{lemma_border_upper_bound}
Let $\{\Omega_{n}\}_{n=1}^{+\infty}$ be a sequence of open convex subdomains  of $\Omega$ 
satisfying (\ref{assmp_subdomains}), and $\{v_{n}\}_{n=1}^{+\infty}$ a sequence of convex functions 
satisfying (\ref{assmp_function_sequence}). Let part $(a)$ of Assumption~\ref{assmp_function_decay} 
and Assumption~\ref{assmp_border} hold. Let $b_{0}$ be the border of $v_{0}$ introduced in 
Assumption~\ref{assmp_function_decay}. Then
\begin{align*}
b_{0}(\bx) \leq \tilde{b}(\bx),\quad \forall \bx \in \partial\Omega. 
\end{align*}
Here, $\tilde{b}$ is a function on $\partial\Omega$ introduced in Assumption~\ref{assmp_border}. 
\end{lemma}

\begin{remark}In the proof \cite[Theorem~$10.6$]{Bakelman94},  it has been stated that the conclusion of Lemma \ref{lemma_border_upper_bound} is trivial. However, we have found that it is really not trivial and the proof needs rather tricky analysis.
\end{remark}

\begin{proof}
We shall prove it by contradiction. By part $(a)$ of Assumption~\ref{assmp_border} and Lemma~\ref{lemma_continuous_border}, 
we can extend $v_{n}$ to $\overline{\Omega_{n}}$ such that 
\begin{align*}
v_{n}\in C^{0}(\overline{\Omega_{n}}),\quad v_{n}|_{\partial\Omega_n} = b_{n},\quad \forall n\in\mathbb{N}. 
\end{align*}

Fixing an $\bx\in \partial\Omega$, then  we can take $\bx^{\prime}\in \partial\Omega \backslash \{ \bx \}$ such that the interior of 
$\overline{\bx\bx^{\prime}}$ is contained in $\Omega$,  
where $\overline{\bx \bx^{\prime}}$ denotes the line segment between $\bx$ and $\bx^{\prime}$. 
By (\ref{assmp_subdomains}), for any $n\in\mathbb{N}$ large enough, there are 
$\bx_{n}, \bx^{\prime}_{n}\in \partial \Omega_{n}$, such that 
\begin{align*}
\overline{\bx_{n}\bx_{n}^{\prime}} \subset \overline{\bx\bx^{\prime}}, \quad 
\overline{\bx_{n}\bx_{n}^{\prime}} \subset \overline{\Omega_{n}},\quad 
\bx_{n}\neq \bx_{n}^{\prime}. 
\end{align*}

We choose an orthogonal coordinate $\{x,x^{2},\cdots, x^{d}\}$ such that $\overline{\bx\bx^{\prime}}$ 
and $\{ \overline{\bx_{n}\bx_{n}^{\prime}}\}$ are all contained in $\{(y,y^{2},\cdots, y^{d})\in\mathbb{R}^{d}: 
y^{2}=\cdots=y^{d}=0\}$. 
Hence, we can use $x, x^{\prime}\in\mathbb{R}$ to represent $\bx$ and $ \bx^{\prime}$ respectively and use 
$x_{n}, x_{n}^{\prime}\in \mathbb{R}$ to represent $\bx_{n}$ and $\bx_{n}^{\prime}$ respectively  for any 
$n\in\mathbb{N}$ large enough. Without  loss of generality, we assume that 
\begin{align*}
x < x_{n}< x_{n}^{\prime}< x^{\prime},\quad \forall n\in\mathbb{N}\text{ large enough}.
\end{align*}
Due to (\ref{assmp_subdomains}), it is easy to see that $x_{n}\ {\rightarrow} \ x\ \text{ as}\  n\rightarrow +\infty. $

We assume $b_{0}(x) > \tilde{b}(x)$ and  define $\epsilon := b_{0}(x) - \tilde{b}(x)$. In the following  we claim that 
\begin{align}
\label{claim_border_converge1}
b_{n}(x_{n}) \ &{\longrightarrow} \ \tilde{b}(x), \hspace{10pt}\text{ as}\  n\rightarrow +\infty. 
\end{align}
In fact ,  $(x_{n}, b_{n}(x_{n}))\in \bS_{n}$ for all $ n\in\mathbb{N}. $
Let $\{x_{n_{k}}\}_{k=1}^{\infty}$ be a subsequence of $\{x_n\}^{\infty}_{n=1}$, such that 
$\displaystyle\lim_{k\rightarrow +\infty} (x_{n_{k}}, b_{n_{k}}(x_{n_{k}})) = (x, \liminf_{n\rightarrow +\infty}b_{n}(x_{n})).$
Due to part $(b)$ of Assumption~\ref{assmp_border} and Definition~\ref{def_topological_limit}, one  obtains that
$\displaystyle( x, \liminf_{n\rightarrow +\infty}b_{n}(x_{n}) ) \in \tilde{\bS}, $ where $\tilde{\bS}$ is the graph of $\tilde{b}\in C^{0}(\partial\Omega)$.
Similarly, we can show that $\displaystyle (x, \limsup_{n\rightarrow +\infty}b_{n}(x_{n})) \in \tilde{\bS}$. Then we get
$\displaystyle\tilde{b}(x) = \liminf_{n\rightarrow +\infty}b_{n}(x_{n}) = \limsup_{n\rightarrow +\infty}b_{n}(x_{n}). $
Therefore (\ref{claim_border_converge1}) is true. 

Due to the definition of $b_{0}$, for any $\delta > 0$, 
there exists some $x^{\prime \prime}\in (x, x+\delta)$  such that 
\begin{align*}\vert v_{0}(x^{\prime \prime}) - b_{0}(x) \vert < \epsilon/6. 
\end{align*}
According to (\ref{assmp_subdomains}), part $(a)$ of Assumption~\ref{assmp_function_decay} 
and (\ref{claim_border_converge1}), there is  $x_{n_{\delta}} \in (x, x^{\prime \prime})\subset (x, x+\delta)$ such that there hold:
\begin{align*}\vert v_{n_{\delta}}(x^{\prime\prime}) - v_{0}(x^{\prime\prime})\vert< \epsilon/6\  \text{and} \ \vert b_{n_{\delta}} (x_{n_{\delta}}) - \tilde{b}(x)\vert < \epsilon/6.\end{align*} 
Since $v_{n}\in C^{0}(\overline{\Omega_{n}})$ and $v_{n}|_{\partial\Omega_{n}}=b_{n}$ for any $n\in\mathbb{N}$, 
then there exists some  $x_{n_{\delta}}^{\prime \prime} \in (x_{n_{\delta}}, x^{\prime\prime}) $ such that \begin{align*}\vert v_{n_{\delta}}(x_{n_{\delta}}^{\prime \prime}) - b_{n_{\delta}} (x_{n_\delta}) \vert < \epsilon/6. 
\end{align*}
By the latest three estimates above, one obtains that
\begin{align}
\label{border_converge2}
\vert v_{n_{\delta}}(x^{\prime\prime}) - b_{0}(x)\vert < \epsilon/3 \quad \text{and}\quad 
\vert v_{n_{\delta}}(x_{n_{\delta}}^{\prime\prime}) - \tilde{b}(x) \vert < \epsilon/3. 
\end{align}
Taking $\tilde{x} := (x + x^{\prime})/2$,  from (\ref{assmp_subdomains}),  we know that
$\tilde{x} \in \Omega_{n}$ for all $n\in \mathbb{N} $ large enough and  $x_{n_{\delta}}^{\prime\prime} < x^{\prime\prime} < \tilde{x}$ for $\delta>0$ small enough.  
By the convexity of $v_{n_{\delta}}$, (\ref{border_converge2}) and the definition of $\epsilon$,  we have 
\begin{align*}
(x^{\prime\prime}-x_{n_{\delta}}^{\prime\prime})v_{n_{\delta}} (\tilde{x}) 
\geq & \hspace{10pt}(\tilde{x} - x_{n_{\delta}}^{\prime\prime}) v_{n_{\delta}}(x^{\prime\prime}) 
- (\tilde{x} - x^{\prime\prime}) v_{n_{\delta}}(x_{n_{\delta}}^{\prime\prime}) \\
\geq &\hspace{10pt} (b_{0}(x) - \epsilon/3)(\tilde{x} - x_{n_{\delta}}^{\prime\prime}) 
- (\tilde{b}(x) + \epsilon/3)(\tilde{x} - x^{\prime\prime}) \\
= &\hspace{10pt} (b_{0}(x) - \epsilon/3)(x^{\prime\prime} - x_{n_{\delta}}^{\prime\prime}) 
+ \epsilon(\tilde{x} - x^{\prime\prime})/3,
\end{align*}
which, together with  the constructions of $x^{\prime\prime}, x_{n_{\delta}}^{\prime\prime}$ and $\tilde{x}$, implies that
\begin{align*}
\dfrac{\tilde{x} - x^{\prime\prime}}{x^{\prime\prime}-x_{n_{\delta}}^{\prime\prime}}
{\longrightarrow} +\infty \hspace{10pt}\text{as}\  {\delta \rightarrow 0}.
\end{align*} 
This leads to $v_{n_{\delta}}(\tilde{x})\rightarrow +\infty$ as $\delta\rightarrow0$, a  contradiction with 
part $(a)$ of Assumption~\ref{assmp_function_decay}. 
\end{proof}

The following Theorem~\ref{thm_border_converge} is a minor revision of \cite[Theorem~$10.6$]{Bakelman94}.
The proof of \cite[Theorem~$10.6$]{Bakelman94} consists of three parts. Its first part is very geometrically intuitive. 
We rewrite the proof with more  detailed explanation in Appendix~\ref{sec_proof_thm_border_converge}.
\begin{theorem}
\label{thm_border_converge}
Let $\{\Omega_{n}\}_{n=1}^{+\infty}$ be a sequence of open convex subdomains  of $\Omega$ 
satisfying (\ref{assmp_subdomains}), and $\{v_{n}\}_{n=1}^{+\infty}$ a sequence of convex functions 
satisfying (\ref{assmp_function_sequence}).
Let Assumptions~\ref{assmp_boundary_para},\ref{assmp_R},\ref{assmp_function_decay},\ref{assmp_border} 
hold for $\Omega$, the function $R$, and $\{v_{n}\}_{n=1}^{+\infty}$. Now let the numbers $k,\lambda$ and 
$\tau$ satisfy:
\begin{eqnarray*} \left\{\begin{array}{lll}
 k&\leq &K \ \text{  if  }\ 0\leq k<1 \ \text{ or } \ k\geq\frac{d}{2},\\
\\
 k&<&K \ \text{  if  } \ 1\leq k \ < \frac{d}{2}
\end{array}
\right.
\end{eqnarray*}
where $K= \dfrac{d+\tau+1}{\tau+2}+\dfrac{\lambda}{2}$. 
Let $b_{0}$ be the border of $v_{0}$ introduced in 
Assumption~\ref{assmp_function_decay} and  $\tilde{b}$ be the function on $\partial\Omega$ introduced in Assumption~\ref{assmp_border}. Then 
\begin{align*}
\tilde{b}(\bx) = b_{0}(\bx),\quad \forall \bx\in \partial\Omega. 
\end{align*}
\end{theorem}
\vspace{20pt}

\section{The finite element method}

In this section, we first introduce the concept of mesh, which is a sequence of convex polyhedra domains 
with standard triangulation to approximate the convex domain $\Omega$. Then, we design 
a finite element method to approximate the exact solution of (\ref{ma_measure_eqs}) and we show that 
this finite element method is well-posed.  

\subsection{The mesh}\
\vspace{5pt}

In this part, we firstly give definition of  the  mesh, which plays an important role in the finite element method. Furthermore, we  show that the convex domain $\Omega$ can be approximated by a sequence of convex polyhedra domains. 

\begin{definition}
\label{def_mesh}
For a given positive real number $h$, we denote by $\mathcal{T}_{h}$ 
a set of $d$-dimensional simplxes contained in $\overline{\Omega}$ such that the following conditions (5.1a) -(5.1d) are fulfilled:
\begin{subequations}
\begin{align}
\label{def_mesh1}
&\text{for any }T,T^{\prime}\in \mathcal{T}_{h} \text{ with} \ T\neq T^{\prime},  
T\cap T^{\prime} \text{ is a } \bar{d} \text{-dimensional sub-simplex of}\nonumber\\
&\text{ both } T  \text{ and }T^{\prime}. \text{ Here}, 0\leq \bar{d}\leq d-1;&\\
\label{def_mesh2}
 &h = \max_{T\in\mathcal{T}_{h}}h_{T}, \ \text{where}\ h_{T} \text{ is  the diameter of } T\in\mathcal{T}_{h};\\
\label{def_mesh3}
 &\Omega_{h} := \text{Int}\big(\overline{\cup_{T\in\mathcal{T}_{h}}T}\big) \ \text{ is a convex domain};\\
\label{def_mesh4}
&\text{  all vertexes} \text{ of } \partial\Omega_{h} \text{ are contained on } \partial\Omega.
\end{align}
\end{subequations}
Then $\mathcal{T}_{h}$ is called a mesh of $\Omega$. 
\end{definition}

\begin{lemma}
\label{lemma_mesh}
For any $\delta>0$, there is a polyhedra domain $P_{\delta}$ such that 
$\overline{\Omega_{\delta}}\subset P_{\delta} \subset \overline{P_{\delta}}\subset \overline{\Omega}$,
 { and all vertexes of } $P_{\delta} $ { are contained on } $\partial\Omega.$ Here $\Omega_{\delta}:=\{\bx\in \Omega: \text{dist}(\bx, \partial\Omega)> \delta\}$ for any $\delta>0$. 

\end{lemma}

\begin{proof}
$\forall\ \epsilon>0$,  we define 
\begin{align*}
C_{\epsilon}:= \{(i_{1}\epsilon, (i_{1}+1)\epsilon ]\times \cdots \times (i_{d}\epsilon, (i_{d}+1)\epsilon ]: 
\forall i_{1},\cdots,i_{d}\in \mathbb{Z}\}.
\end{align*}
Obviously, $\mathbb{R}^{d}$ can be covered by cubes in $C_{\epsilon}$ without overlapping. 
For $0<\epsilon \leq \frac{1}{3\sqrt{d}}\delta$, there exist finitely many $\{\boldsymbol{C}_{i}\}_{i=1}^{m}\subset C_{\epsilon}$ such that $\boldsymbol{C}_i\cap \partial\Omega\neq \emptyset$ for any $i=1,\cdots, m$. 
Taking  any point, denoted by $\bB_{i}$,  in $\boldsymbol{C}_{i}\cap \partial\Omega$ for any $1\leq i\leq m$,  we define  $P_{\delta}$ to be the convex hull of $\{\bB_{i}\}_{i=1}^{m}$. Then all vertexes of $P_{\delta}$ 
stay on $\partial\Omega$ and $\overline{P_{\delta}}\subset \overline{\Omega}$ since $\Omega$ is convex. 

In the following, we shall show that $\overline{\Omega_{\delta}}\subset P_{\delta} $. In fact, if not, then there is a point $\bx_{0}\in \overline{\Omega_{\delta}}$ and $\bx_{0}\notin P_{\delta}$. 
Without loss of generality, we may assume that  
$$\bx_{0} = (0,\cdots, 0)\in \mathbb{R}^{d} \text{ and }  x^{d}<0 \text{ for any }(x^{1},\cdots,x^{d})=\bx\in P_{\delta}. 
$$
Since $\overline{\Omega_{\delta}}\subset \Omega$, there is $x^{d}>0$ such that the point 
$\bx^{\prime}:=(0,\cdots, 0, x^{d})\in \partial\Omega$. Due to the definition of $\Omega_{\delta}$, 
$\text{dist}(\bx_{0}, \bx^{\prime})= x^{d}\geq \delta$. 
Obviously, there is a cube $\boldsymbol{C}^{\prime}\in C_{\epsilon}$ such that $\bx^{\prime}\in \boldsymbol{C}^{\prime}$. 
Then there is some positive integer $j$  with $1\leq j\leq m$  such that   $\bB_{j}\in \boldsymbol{C}^{\prime}$ and $\text{dist}(\bx^{\prime}, \bB_{j})\leq \sqrt{d}\epsilon\leq  \delta/3$, which implies that 
\begin{align*}
x_{j}^{d}\ \geq \ {2}\delta/3\ >\ 0\text{ where } (x_{j}^{1},\cdots,x_{j}^{d})= \bB_{j}.
\end{align*}
This contradics with the fact that $\bB_j\in P_\delta$. Therefore $\overline{\Omega_{\delta}}\subset P_{\delta}$.
\end{proof}

According to Lemma~\ref{lemma_mesh} and standard triangulation for polyhedra, 
there is $I\subset (0,1)$ such that the following conditions are fulfilled:
\begin{subequations}
\label{mesh_props}
\begin{align}
\label{mesh_prop1}
& 0\text{ is the unique accumulation point of }I;\\
\label{mesh_prop2}
& \text{for any } h\in I, \text{ there is a mesh }\mathcal{T}_{h} \text{ of }\Omega;\\
\label{mesh_prop3}
& \text{for any }\delta >0, \text{ there is }h_{\delta}>0 \text{ such that }
\overline{\Omega_{\delta}}\subset \Omega_{h}\text{ if } h\in I \text{ and } h<h_{\delta}. 
\end{align}
\end{subequations}
In fact, the proof of Lemma~\ref{lemma_mesh} is constructive such that 
it naturally provides an algorithm to construct the convex polyhedra to approximate $\Omega$.

\subsection{The finite element method}\
\vspace{5pt}

For any given mesh $\mathcal{T}_{h}$ of $\Omega$, we denote  
the vertexes of $\mathcal{T}_{h}$ contained in the interior of $\Omega_{h}$ and 
 the vertexes of $\partial\Omega_{h}$ by $\{\bA_{i}\}_{i=1}^{k_{h}}$ and $\{\bB_{j}\}_{j=1}^{m_{h}}$, respectively, and we define $M_{h}(z_{1},\cdots,z_{k_{h}})$ to  be the convex hull of 
$\{(\bA_{i}, z_{i})\}_{i=1}^{k_{h}} \cup \{(\bB_{j}, g(\bB_{j}))\}_{j=1}^{m_{h}}$  in $\mathbb{R}^{d+1}$
for any real numbers $\{z_{i}\}_{i=1}^{k_{h}}$. We introduce  $K_{h}: = \{M_{h}(z_{1},\cdots, z_{k_{h}}): \forall z_{i}\in\mathbb{R}, 1\leq i\leq k_{h}\}$ and
\begin{align}
\label{def_H}
H_{h} := \{ v\in W^{+}(\Omega_{h})\cap C(\overline{\Omega_{h}}): \exists M_{h}\in K_{h} \text{ such that }
v(\bx) = \inf_{(\bx,z)\in M_{h}}z,\forall \bx \in \overline{\Omega_{h}} \}.
\end{align}

The  so-called finite element method is to find $u_{h}\in H_{h}$ such that 
\begin{align}
\label{fem_eqs}
\int_{\partial u_{h}(\bA_{i})}R(\bp) d\bp = \int_{\Omega_{h}}\phi_{i,h}d\mu,\quad \forall 1\leq i\leq k_{h},
\end{align} 
where , for any $1\leq i\leq k_{h}$, $\phi_{i,h}\in C_{c}(\Omega_{h})\cap \mathbb{P}_{1}(\mathcal{T}_{h})$  with the conditions:
\begin{align}
\label{fem_node_func}
\phi_{i,h}(\bA_{j}) = \delta_{ij},\quad\forall 1\leq j\leq k_{h}
\end{align} and $\mathbb{P}_{1}(\mathcal{T}_{h})$ is defined to be the set of piecewise linear functions on $\mathcal{T}_h$.

\begin{definition}
A domain $\Omega\subset \mathbb{R}^{d}$ is called strictly convex if for any $\bx,\bx^{\prime}\in \overline{\Omega}$, 
\begin{align*}
\lambda \bx + (1-\lambda)\bx^{\prime}\in \Omega,\quad \forall 0< \lambda < 1.
\end{align*} 
\end{definition}

\begin{lemma}
\label{lemma_H_B}
 Assume that $\Omega$ is strictly convex and  $\mathcal{T}_{h}$ is a mesh of $\Omega$. Let
$\{\bA_{i}\}_{i=1}^{k_{h}}$ and $\{\bB_{j}\}_{j=1}^{m_{h}}$ are the vertexes of $\mathcal{T}_{h}$ contained in the interior of 
$\Omega_{h}$ and the vertexes of $\partial\Omega_{h}$, respectively. Then there hold:
\begin{subequations}
\label{H_B_props}
\begin{align}
\label{H_exist}
& H_{h}\ \neq \ \emptyset;\\
\label{H_B}
& v(\bB_{j}) = g(\bB_{j}),\quad \forall v\in H_{h}  \ \text{and} \ 1\leq j\leq m_{h};\\
\label{H_boundary}
& w = v\text{ on }\partial\Omega_{h},\quad \forall w,v\in H_{h}.
\end{align}
\end{subequations}
\end{lemma}

\begin{proof}
For any $1\leq i\leq k_{h}$, we take
\begin{align*}
z_{i} = \big( \max_{1\leq j\leq m_{h}}g(\bB_{j}) \big) +1.
\end{align*}
Then $M_{h}:=M_{h}(z_{1},\cdots,z_{k_{h}})$ is the convex hull of $\{\bB_{j}\}_{j=1}^{m_{h}}$. We define
\begin{align*}
v(\bx) = \inf_{(\bx,z)\in M_{h}}z,\quad \forall \bx \in \overline{\Omega_{h}}. 
\end{align*}
Obviously, $v\in W^{+}(\Omega_{h})\cap C(\overline{\Omega_{h}})$. Thus $v\in H_{h}$ , which
shows that $H_{h}\neq \emptyset$.

Since $\Omega$ is strictly convex and (\ref{def_mesh4}) holds true for any $1\leq j \leq m_{h}$, 
$\bB_{j}$ is not contained in the convex hull of $\{(\bA_{i}, z_{i})\}_{i=1}^{k_{h}} \cup \{(\bB_{l}, 
g(\bB_{l}))\}_{l=1,l\neq j}^{m_{h}}$. Then we obtain (\ref{H_B}). Finally from (\ref{H_B}) and  (\ref{def_H}),  the statement (\ref{H_boundary}) holds true.
\end{proof}

\begin{assumption}
\label{assmp_source_integral}
We assume that 
\begin{align*}
\int_{\Omega} d\mu < \int_{\mathbb{R}^{d}}R(\bp)d\bp.
\end{align*}
\end{assumption}

\begin{theorem}
\label{thm_fem_unique}
Let $\mathcal{T}_{h}$ be a mesh of $\Omega$. 
We assume that $\Omega$ is strictly convex and Assumption~\ref{assmp_source_integral} holds. 
Then, the finite element method (\ref{fem_eqs}) has a unique solution. 
\end{theorem}

\begin{proof} By Assumption~\ref{assmp_source_integral}, we know that 
\begin{align}
\label{discrete_source_integral}
\displaystyle\sum_{i=1}^{k_{h}}\int_{\Omega_{h}}\phi_{i,h}d\mu= \int_{\Omega_{h}}\big(\sum_{i=1}^{k_{h}}\phi_{i,h}\big)d\mu 
\leq \int_{\Omega_{h}}d\mu \leq \int_{\Omega}d\mu < \int_{\mathbb{R}^{d}}R(\bp)\ d\bp.
\end{align}

Now we replace the set $H$ in the proof of \cite[Theorem~$11.1$]{Bakelman94} by
\begin{align*}
\{v\in H_{h}: \int_{\partial v(\bA_{i})}R(\bp )\ d\bp \leq \int_{\Omega_{h}}\phi_{i,h}d\mu,\quad \forall 1\leq i\leq k_{h}\},
\end{align*}
where  $H_{h}$ is introduced  in (\ref{def_H}). 
By (\ref{H_B_props}, \ref{discrete_source_integral}), the proof 
of  \cite[Theorem~$11.1$]{Bakelman94} can go through, such that we can conclude that 
the finite element method (\ref{fem_eqs}) has a unique solution.
\end{proof}

\vspace{20pt}

\section{Convergence of the finite element method to (1.1)}

In this section, we show that under suitable assumptions, (\ref{ma_measure_eqs}) is well-posed and 
the solutions of the finite element method (\ref{fem_eqs}) converges to the exact solution. 
Theorem~\ref{thm_converge1} is the main result. 
\subsection{Convergence of border of solutions of the finite element method}

\

Before we prove the convergence of solutions of the finite element method (\ref{fem_eqs}), we firstly give Lemma 6.1 and 6.2, which 
show the convergence of border of finite element solutions.

\begin{lemma}
\label{lemma_diam_vanish}
Let $I\subset (0,1)$ satisfy (\ref{mesh_prop1},\ref{mesh_prop2},\ref{mesh_prop3}) and  $\Sigma_{h}$ be the set of all $(d-1)$-dimensional closed polyhedra on $\partial\Omega_{h}$. 
If $\Omega$ is strictly convex, then 
\begin{align*}
\lim_{I\ni h\rightarrow 0}\sup_{\boldsymbol{K}\in \Sigma_{h}}\big( \sup_{\bx,\bx^{\prime}\in \boldsymbol{K}}
\vert \bx - \bx^{\prime} \vert \big) \ =\  0.
\end{align*}
\end{lemma}

\begin{proof}
We prove it by contradiction. If Lemma~{\ref{lemma_diam_vanish}} does not hold true, then  there is $\{h_{n}\}_{n=1}^{+\infty}\subset I$ such that $\displaystyle\lim_{n\rightarrow +\infty}h_{n}=0$
and for any $n\in \mathbb{N}$, there exists some $\boldsymbol{K}_{n}\in \Sigma_{h_{n}}$ such that the following condition holds true:
\begin{align*}
\sup_{\bx,\bx^{\prime}\in \boldsymbol{K}_{n}}\vert \bx - \bx^{\prime} \vert \geq \epsilon_{0}
\end{align*}for some positive constant $\epsilon_{0}$.  For any $n\in \mathbb{N}$, 
since $\boldsymbol{K}_{n}$ is compact, then there are two vertexes $\bx_{n}^{\prime}, \bx_{n}^{\prime \prime}$ of 
$\boldsymbol{K}_{n}$ such that 
\begin{align*}
\vert \bx_{n}^{\prime} - \bx_{n}^{\prime \prime}\vert = \sup_{\bx,\bx^{\prime}\in \boldsymbol{K}_{n}}
\vert \bx - \bx^{\prime} \vert \geq \epsilon_{0}.
\end{align*}
By (\ref{def_mesh4}), $\bx_{n}^{\prime}, \bx_{n}^{\prime \prime} \in \partial\Omega$, for any $n\in \mathbb{N}$.
Without loss of generality (we always can take a subsequence of $\{h_{n}\}_{n=1}^{+\infty}$ if necessary), we have that 
\begin{align*}
\lim_{n\rightarrow +\infty}\bx_{n}^{\prime} = \bar{\bx}^{\prime}\in \partial \Omega \quad \text{and}\ 
\lim_{n\rightarrow +\infty}\bx_{n}^{\prime \prime} = \bar{\bx}^{\prime \prime}\in \partial \Omega.
\end{align*} 
Then by the latest two estimates above, we can see that $\vert \bar{\bx}^{\prime} - \bar{\bx}^{\prime \prime} \vert\geq \epsilon_{0}>0$. 
Since $\Omega$ is strictly convex, then
$
\lambda \bar{\bx}^{\prime} + (1-\lambda)\bar{\bx}^{\prime \prime} \in \Omega, \ \forall \ 0< \lambda < 1.
$
By the definition of $\Omega_{\delta}$, we can choose $\delta>0$ small enough such that 
$
\lambda \bar{\bx}^{\prime} + (1-\lambda)\bar{\bx}^{\prime \prime} \in \text{Int}(\Omega_{\delta}),
$ for all $1/3< \lambda < {2}/{3}.$
Then, by (\ref{mesh_prop3}), we get that
$\big(\bx_{n}^{\prime} + \bx_{n}^{\prime \prime}\big)/2\in \text{Int}(\Omega_{\delta})\subset \Omega_{h_{n}}
\text{ if }n \text{ is large enough}$, which arrives at a contradiction since $\bx_{n}^{\prime}, \bx_{n}^{\prime \prime}$ are two vertexes of $\boldsymbol{K}_{n}$ and $\boldsymbol{K}_{n}\in \Sigma_h\subset\partial \Omega_h$ is convex.
\end{proof}

\begin{lemma}
\label{lemma_border_converge}
Let $I\subset (0,1)$ satisfy (\ref{mesh_prop1},\ref{mesh_prop2},\ref{mesh_prop3}) and  $\Omega$ be strictly convex.  We define 
\begin{align*}
& \bS_{0}:=\{(\bx,g(\bx)): \forall \bx \in \partial\Omega\} \text{and} \  \bS_{h}:=\{(\bx, v(\bx)):\forall \bx \in \partial\Omega_{h}\}, \forall v\in H_{h}.
\end{align*}
Then
$\bS_{h}$ is independent of the choice of $v\in H_{h}$  and it is a $(d-1)$-dimensional surface homeomorphic to the $(d-1)$-unit sphere. Furthermore, we have  that $\lim_{I\ni h\rightarrow 0}^{T}\bS_{h} = \bS_{0}$.
\end{lemma}

\begin{proof}
According to (\ref{H_boundary}),  it is easy to see that $\bS_{h}$ is independent of the choice of $v\in H_{h}$ and it is a $(d-1)$-dimensional surface homeomorphic to the $(d-1)$-unit sphere. 

In the following, we prove that
$\lim_{I\ni h\rightarrow 0}^{T}\bS_{h} = \bS_{0}.$ Firstly we define a function $g_{h}: \partial\Omega_{h}\longrightarrow \mathbb{R}$ by 
$(\bx, g_{h}(\bx))\in \bS_{h},\quad \forall \bx \in \partial\Omega_{h}.$ We take $\bx_0\in \partial\Omega$ arbitrarily, and for any $h\in I$, we define
$\bx_{h}$ to be a vertex on $\partial\Omega_{h}$ which reaches 
the shortest distance between $\bx_0$ and all vertexes ($\{\bB_{j}\}_{j=1}^{m_{h}}$) on $\partial\Omega_{h}$. 
Obviously, $\displaystyle\lim_{I\ni h\rightarrow 0}\bx_{h} = \bx_0$. 
Since $g\in C(\partial\Omega)$, $\displaystyle\lim_{I\ni h\rightarrow 0}(\bx_{h},g(\bx_{h})) = (\bx,g(\bx))$, which implies that 
\begin{align*}
\bS_{0} \subset \underline{\lim}^{T}_{I\ni h\rightarrow 0}\bS_{h}.
\end{align*}
So, it is sufficient to show that 
\begin{align*}
\overline{\lim}^{T}_{I\ni h\rightarrow 0}\bS_{h}\subset \bS_{0}. 
\end{align*}

We take $\{h_{n}\}_{n=1}^{+\infty}\subset I$ such that $\displaystyle\lim_{n\rightarrow +\infty}h_{n}=0$
and choose $\{\bx_{n}\}^{+\infty}_{n=1}\subset \partial\Omega_{h_{n}}$ such that $\displaystyle\lim_{n\rightarrow +\infty}(\bx_{n},g_{h_{n}}(\bx_{n}))$ exists.
By (\ref{mesh_prop3}),we know that  $\displaystyle\lim_{n\rightarrow +\infty}\bx_{n}\in \overline{\Omega}\setminus \Omega$. Thus we obtain that
\begin{align*}
\lim_{n\rightarrow +\infty}(\bx_{n},g_{h_{n}}(\bx_{n})) = (\bx_{0}, z_{0}) \ \ \text{ where }\bx_{0}\in \partial\Omega\ \text{and}\ \ z_{0}\in\mathbb{R}. 
\end{align*}

In the following, we show that $z_{0} = g(\bx_{0})$. $\forall \ n\in \mathbb{N}$, 
there is a $(d-1)$-dimensional closed polyhedra $\boldsymbol{K}_{n}$ such that 
$\bx_{n}\in \boldsymbol{K}_{n}\subset \partial\Omega_{h_{n}}.$ We denote by $\{\bB_{i,n}\}_{i=1}^{l_{n}}$ all vertexes of $\boldsymbol{K}_{n}$ ($l_{n}$ may not  
have a uniform bound).  For  $\boldsymbol{K}_n$, since  Lemma~\ref{lemma_diam_vanish} holds and $\bx_{n}\rightarrow \bx_{0}$ as $n\rightarrow+\infty$, we have that 
\begin{align*}
\lim_{n\rightarrow+\infty}\sup_{\bx,\bx^{\prime}\in \boldsymbol{K}_{n}}\vert \bx - \bx^{\prime} \vert \ = \ \lim_{n\rightarrow+\infty}\text{dist}(\bx_{0}, \boldsymbol{K}_{n})\ =\ 0 
\end{align*}
which, together with ({\ref{H_B}}) and $g\in C(\partial \Omega)$,  implies that
\begin{align*}
\max_{1\leq i\leq l_{n}}\vert g_{h_{n}}(\bB_{i,n})-g(\bx_{0})) \vert\rightarrow 0\quad \text{ as } n\rightarrow +\infty.
\end{align*}
Since $\bx_{n}\in \boldsymbol{K}_{n}$, there are nonnegative numbers $\{ \lambda_{i,n},\}_{1\leq i\leq l_n}$ such that there hold:
\begin{align*}\left\{\begin{array}{lll}
 &\lambda_{1,n}+\cdots + \lambda_{l_{n},n} = 1,\\ \\
 & \lambda_{1,n}\bB_{1,n}+\cdots + \lambda_{l_{n}, n}\bB_{l_{n},n} = \bx_{n},\\ \\
 & \lambda_{1,n} g_{h_{n}}(\bB_{1,n})+\cdots + \lambda_{l_{n},n} g_{h_{n}}(\bB_{l_{n},n}) = g_{h_{n}}(\bx_{n}).
\end{array}
\right.
\end{align*}Here the third equality follows from the definitions of ${\boldsymbol K}_n$, ${\bS}_h$ and $g_h$.
Then,  it holds:
\begin{align*}
 \vert g_{h_{n}}(\bx_{n}) - g(\bx_{0})\vert 
&\leq  \sum_{i=1}^{l_n}\lambda_{i, n}\vert g_{h_{n}}(\bB_{i,n}) - g(\bx_{0}) \vert\\
&\leq  \max_{1\leq i\leq l_{n}}\vert g_{h_{n}}(\bB_{i,n})-g(\bx_{0})) \vert
\rightarrow 0 \quad \text{ as }n\rightarrow +\infty.
\end{align*}
Therefore, $z_{0}= g(\bx_{0})$. 
\end{proof}
\vspace{10pt}

\subsection{Convergence of finite element method }

\begin{lemma}
\label{lemma_fem_bound}
Assume $\Omega$ is strictly convex. If Assumption~\ref{assmp_source_integral} holds and  $I\subset (0,1)$ satisfies (\ref{mesh_prop1},\ref{mesh_prop2},\ref{mesh_prop3}),
then there exists a uniform constant $M>0$ such that for any $h\in I$,   the solution of the finite element method (\ref{fem_eqs}), denoted by $u_{h}$, satisfies :
\begin{align*}
\Vert u_{h}\Vert_{L^{\infty}(\Omega_{h})} \leq M, \quad \forall h\in I. 
\end{align*}
\end{lemma}

\begin{proof}
By the construction of $\{u_{h}\}_{h\in I}$ in (\ref{fem_eqs}) and Lemma~\ref{lemma_H_B}, 
it holds for any $n\in \mathbb{N}$, 
\begin{align*}
\min_{\by\in\partial\Omega} g(\by) \ \leq \ u_{h}(\bx) \leq \max_{\by\in\partial\Omega} g(y),
\quad \forall \bx\in \partial\Omega_{h}.
\end{align*}
Since $u_{h}\in W^{+}(\Omega_{h})$, then we can derive that 
\begin{align*}
u_{h}(\bx) \leq \max_{\by\in\partial\Omega}g(\by),\quad \forall \bx\in \Omega_{h}.
\end{align*}
In the following, we will deduce some lower bound for $u_{h}$ in $\Omega_{h}$. 
Let $\bx_{h}\in \Omega_{h}$ such that 
\begin{align*}
u_{h}(\bx_{h}) = \min_{\bx\in \overline{\Omega_{h}}} u_{h}(\bx). 
\end{align*}
Without loss of generality, we assume 
$u_{h}(\bx_{h}) < \min_{\by\in \partial\Omega} g(\by)$.  We define 
\begin{align*}
\rho_{h} :=\dfrac{\min_{\by\in\partial\Omega}g(\by) - u_{h}(\bx_{h})}
{\sup_{\bx,\bx^{\prime}\in \Omega}\vert \bx - \bx^{\prime} \vert}.
\end{align*}
Then it is easy to see that 
$\overline{B_{\rho_{h}}(\boldsymbol{0})} \subset \partial u_{h}(\Omega_{h}) \subset \mathbb{R}^{d}, $
which implies that
\begin{align*}
\int_{B_{\rho_{h}}(\boldsymbol{0})} R(\bp) d\bp \leq \int_{\partial u_{h}(\Omega_{h}) } R(\bp) d\bp.
\end{align*}
By the construction of $u_{h}$ and Assumption~\ref{assmp_source_integral}, we know that 
\begin{align*}
\int_{\partial u_{h}(\Omega_{h}) } R(\bp) d\bp = \int_{\Omega_{h}} \big( \sum_{i=1}^{k_{h}} \phi_{i,h} \big) d\mu
\leq \int_{\Omega_{h}} d\mu \leq  \int_{\Omega} d\mu < \int_{\mathbb{R}^{d}} R(\bp) d\bp.
\end{align*}
Then  by combining the latest two estimates above, it holds:
\begin{align*}
\int_{B_{\rho_{h}}(\boldsymbol{0})} R(\bp) d\bp \leq \int_{\Omega} d\mu < \int_{\mathbb{R}^{d}} R(\bp) d\bp.
\end{align*}
We set
$
\displaystyle g_{R}(\rho) := \int_{B_{\rho}(\boldsymbol{0})} R(\bp)  d\bp \quad \text{for all } \rho>0 \ \text{and}  \ \omega_{0} : = \int_{\Omega} d\mu.
$
Obviously, $g_{R}: [0,+\infty) \rightarrow [0,+\infty)$ is strictly increasing and 
$g_{R}^{-1}$ exists (it is also strictly increasing and continuous). Then we infer that
\begin{align*}
g_{R}(\rho_{h}) \leq \omega_{0} < \int_{\mathbb{R}^{d}} R(\bp) d\bp, 
\end{align*}
which implies that $0 < \rho_{h} \leq g_{R}^{-1}(\omega_{0}) < +\infty. $ Hence by the definition of $g_h$ , we get
\begin{align*}
u_{h}(\bx_{h}) \geq \min_{\by\in \partial \Omega}g(\by) 
- \big( \sup_{\bx,\bx^{\prime}\in \Omega}\vert \bx - \bx^{\prime} \vert\big)\cdot g_{R}^{-1}(\omega_{0}). 
\end{align*}
Therefore, for any $h\in I$, it holds:
\begin{align*}
\min_{\by\in \partial \Omega}g(\by) 
- \big( \sup_{\bx,\bx^{\prime}\in \Omega}\vert \bx - \bx^{\prime} \vert\big)\cdot g_{R}^{-1}(\omega_{0})
\leq u_{h}(\bx) \leq \max_{\by\in\partial\Omega} g(\by), \quad \forall \bx \in \Omega_{h}. 
\end{align*}
\end{proof}

\begin{assumption}
\label{assmp_source_decay}
For any $\bx_{0}\in \partial\Omega$, there exists an open $d$-ball $U_{\rho}(\bx_{0})$ such that 
\begin{align*}
\int_{e\cap \Omega} d\mu \leq C_{1}^{\prime}\big( \sup_{\bx \in e} 
 \text{dist}(\bx, \partial\Omega)\big)^{\lambda}\vert e\vert,
\quad \forall \text{ Borel set } e\subset U_{\rho}(\bx_{0})\cap \Omega.
\end{align*}
Here, $C_{1}^{\prime}>0$ and $\lambda \geq 0$ are constants 
independent of the choice of $\bx_{0}\in\partial\Omega$. 
\end{assumption}

\begin{lemma}
\label{lemma_border_valid}
Let Assumptions~\ref{assmp_boundary_para},\ref{assmp_R},\ref{assmp_source_integral},\ref{assmp_source_decay}  
hold  and  the numbers $k,\lambda$ and $\tau$ satisfy 
\begin{align*}\left\{\begin{array}{lll}
& k\leq K \text{  if  }0\leq k<1 \text{ or } k\geq {d}/{2},\\ \\
& k<K \text{  if  } 1\leq k < {d}/{2},
\end{array}
\right.
\end{align*}
where $K= \dfrac{d+\tau+1}{\tau+2}+\dfrac{\lambda}{2}$. 
If $I\subset (0,1)$ satisfies (\ref{mesh_prop1},\ref{mesh_prop2},\ref{mesh_prop3}) and  $u_{h}$ is the solution of the finite element method (\ref{fem_eqs}) for any $h\in I$, 
then there exist  a sequence $\{h_{n}\}_{n=1}^{+\infty}\subset I$ with $\displaystyle\lim_{n\rightarrow +\infty}h_{n} = 0$, 
and $u_{0}\in W^{+}(\Omega)\cap C(\overline{\Omega})$ such that  $u_0$ solves (1.1) and 
\begin{align*}
& \lim_{n\rightarrow + \infty}\Vert u_{h_{n}} - u_{0}\Vert_{L^{\infty}(\overline{\Omega_{\delta}})}  = 0,
\quad \forall \delta > 0,
\end{align*}where $\Omega_{\delta}= \{\bx\in\Omega: \text{dist}(\bx,\partial\Omega)> \delta\}$.
\end{lemma}

\begin{proof}
According to (\ref{mesh_prop3}) and Lemma~\ref{lemma_fem_bound}, we can apply 
Theorem~\ref{thm_convex_sequence_converge_domain} to $\{u_{h}\}_{h\in I}$. 
Therefore, by Theorem~\ref{thm_convex_sequence_converge_domain},  
there exist a sequence $\{h_{n}\}_{n=1}^{+\infty}\subset I$ with 
$\displaystyle\lim_{n\rightarrow +\infty}h_{n} = 0$, and a function $u_{0}\in W^{+}(\Omega)$ such that
\begin{align*}\left\{\begin{array}{lll}
& \displaystyle\lim_{n\rightarrow + \infty}\Vert u_{h_{n}} - u_{0}\Vert_{L^{\infty}(\overline{\Omega_{\delta}})}  = 0,
\quad \forall \delta > 0,\\ \\
& \displaystyle\lim_{n \rightarrow +\infty}\int_{\Omega_{h_{n}}}f d\mu_{n} = 
\int_{\Omega}f d\mu_{0},\quad \forall f\in C_{c}^{0}(\Omega),
\end{array}
\right.
\end{align*}
where $\mu_{n},\mu_{0}$ are measures in $\Omega_{h_{n}}, \Omega$ defined as
\begin{align*}\left\{\begin{array}{lll}
\mu_{n}(e) = & \int_{\partial u_{h_{n}}(e)} R(\bp) d\bp,\quad \forall \text{ Borel set } e\subset \Omega_{h_{n}},\\ \\
\mu_{0}(e) = & \int_{\partial u_{0}(e)} R(\bp) d\bp,\quad \forall \text{ Borel set } e\subset \Omega.
\end{array}
\right.
\end{align*}

From the construction of $u_{h}$ in (\ref{fem_eqs}), 
we know that for any $f\in C_{c}(\Omega)$,  
\begin{align*}
\int_{\Omega_{h_{n}}} f d\mu_{n} = \sum_{i=1}^{k_{h_{n}}}f(\bA_{i})\int_{\Omega_{h_{n}}}\phi_{i,h_{n}}d \mu
=\int_{\Omega_{h_{n}}} \sum_{i=1}^{k_{h_{n}}}\big(f(\bA_{i})\phi_{i,h_{n}}\big) d \mu, 
\end{align*}
where $\{A_{i}\}_{i=1}^{k_{h_{n}}}$ are the vertexes of $\mathcal{T}_{h_{n}}$ contained in the interior of $\Omega_{h_{n}}$. 
By (\ref{mesh_prop3}) and the construction of $\phi_{i,h}$ in (\ref{fem_node_func}), it is easy to see that 
\begin{align*}
\displaystyle\lim_{n\rightarrow +\infty}\sum_{i=1}^{k_{h_{n}}}\big(f(\bA_{i})\phi_{i,h_{n}}\big)(\bx)
=f(\bx)\  \forall \bx\in \Omega,\ \text{and} \ \displaystyle\sup_{\bx\in \Omega_{h_{n}}}\vert \sum_{i=1}^{k_{h_{n}}}\big(f(\bA_{i})\phi_{i,h_{n}}\big)(\bx)\vert 
\leq \sup_{x\in \Omega} \vert f(\bx) \vert. 
\end{align*}
Then by dominated convergence Theorem, we know that 
\begin{align*}
\int_{\Omega_{h_{n}}}f d\mu_{h} \rightarrow \int_{\Omega} f d\mu, \text{ as } n\rightarrow +\infty.
\end{align*}
This implies that 
\begin{align*}
\int_{\Omega} f d\mu_{0} = \int_{\Omega} f d\mu,\quad \forall f\in C_{c}^{0}(\Omega). 
\end{align*}
Thus, we have that 
\begin{align*}
\int_{\partial u_{0}(e)}R(\bp)d\bp = \mu_{0}(e) = \mu (e)\quad \forall \text{  Borel set  } e\subset\Omega. 
\end{align*}

We denote by $b_{0}$ the border of $u_{0}$, which is given by
\begin{align*}
b_{0}(\bx) = \liminf_{\Omega\ni\bx^{\prime} \rightarrow \bx} u_{0}(\bx^{\prime}),
\quad \forall \bx \in \partial\Omega. 
\end{align*}
By Lemma~\ref{lemma_continuous_border} and the fact that $g\in C(\partial\Omega)$, it is sufficient to show
\begin{align}
\label{border_eq}
b_{0} = g \text{  on  } \partial\Omega.
\end{align}
In fact, (\ref{border_eq}) would be an immediate consequence if we apply Theorem~\ref{thm_border_converge} 
to $\{u_{h_{n}}\}_{n=1}^{+\infty}$ and $u_{0}$. In the following, we only need to verify all assumptions 
of Theorem~\ref{thm_border_converge} hold. Obviously, from our assumptions, Assumptions~\ref{assmp_boundary_para} and \ref{assmp_R} hold. By (\ref{fem_node_func}) and the construction of $u_{h}$ in (\ref{fem_eqs}), 
we know that for any $n\in \mathbb{N}$, 
\begin{align*}
\int_{\partial u_{h_{n}}(e\cap \Omega_{h_{n}})}R(\bp) d\bp 
= \int_{e\cap \Omega_{h_{n}}}\big( \sum_{i=1}^{k_{h_{n}}}\phi_{i,h_{n}} \big) d\mu \leq\mu (e),
\quad \forall \text{ Borel set } e \subset \Omega.
\end{align*}
By Assumption~\ref{assmp_source_decay}, it is easy to check that Assumption~\ref{assmp_function_decay} 
holds for $\{u_{h_{n}}\}_{n=1}^{+\infty}$ and $u_{0}$. Moreover, by Lemma~\ref{lemma_border_converge} 
and the construction of $u_{h}$ in (\ref{fem_eqs}), Assumption~\ref{assmp_border} is also valid 
for $\{u_{h_{n}}\}_{n=1}^{+\infty}$ and $u_{0}$. Thus all the assumptions of Theorem~\ref{thm_border_converge} 
hold. Then we have (\ref{border_eq}).
\end{proof}

\begin{theorem}
\label{thm_converge1}
Let all the assumptions of Lemma~{\ref{lemma_border_valid}} hold. Then, (1.1) admits a unique function $u\in W^{+}(\Omega)\cap C(\overline{\Omega})$. In addition, for any $\delta>0$, there holds:
\begin{align}
\label{compact_uniform_conv}
\lim_{I\ni h\rightarrow 0}\Vert u_{h} - u\Vert_{L^{\infty}(\overline{\Omega_{\delta}})}=0
\end{align}
where  $\Omega_{\delta}= \{\bx\in\Omega: \text{dist}(\bx,\partial\Omega)> \delta\}$ and for any $h\in I$, $u_h$ is the solution of the finite element method (5.4).
\end{theorem}

\begin{proof}
By Lemma~\ref{lemma_border_valid}, we know that (1.1) admits a solution
$u_{0}\in W^{+}(\Omega)\cap C(\overline{\Omega})$. By Theorem~\ref{thm_comparison}, we know that $u_{0}$ is the unique solution  to (1.1).  

We shall prove  (\ref{compact_uniform_conv}) by contradiction. If (\ref{compact_uniform_conv}) is not true, then there is $\delta_{0}>0$ and 
$\{h_{n}^{\prime}\}_{n=1}^{+\infty}\subset I$ with $\displaystyle\lim_{n\rightarrow +\infty}h_{n}^{\prime}=0$, such that 
\begin{align*}
\lim_{n\rightarrow +\infty}\Vert u_{h_{n}^{\prime}} - u_{0}\Vert_{L^{\infty}(\overline{\Omega_{\delta_{0}}})}\neq 0.
\end{align*}
By applying Lemma~\ref{lemma_border_valid} to $\{u_{h_{n}^{\prime}}\}_{n=1}^{+\infty}$,  we know that
there exist a function $u_{0}^{\prime}\in W^{+}(\Omega)\cap C(\overline{\Omega})$ satisfying (1.1), 
and  a subsequence of $\{u_{h_{n}^{\prime}}\}_{n=1}^{+\infty}$, still denoted by  $\{u_{h_{n}^{\prime}}\}_{n=1}^{+\infty}$, such that $u_{h_{n}^{\prime}}$ converges to $u_{0}^{\prime}$ uniformly on 
$\overline{\Omega_{\delta_{0}}}$ as $n\rightarrow +\infty$.  Since (1.1) admits a unique solution, then $u_0=u_0^{\prime}$. This is a contradiction. Therefore, (\ref{compact_uniform_conv}) is true.
\end{proof}

\vspace{20pt}

\section{Generalized solution with Dirichlet data imposed weakly}

In this section, we firstly  introduce the finite element method (\ref{fem_weak_eqs}) for solving (\ref{weak_eqs}), 
which is based on the finite element method (\ref{fem_eqs}). Then we show that (\ref{weak_eqs}) is well-posed 
and the solutions of (\ref{fem_weak_eqs}) converge to the exact solution. The main result in this section is Theorem~\ref{thm_converge2}. 

The finite element method for (\ref{weak_eqs}) is to find $u_{h}^{\delta}\in H_{h}$ such that 
\begin{align}
\label{fem_weak_eqs}
\int_{\partial u_{h}^{\delta}(\bA_{i})}R(\bp) d\bp = \int_{\Omega_{h}}\phi_{i,h}d\mu^{\delta},
\quad \forall 1\leq i\leq k_{h}.
\end{align} 
Here,  for any $\delta>0$, $\mu^{\delta}$ is a measure defined by $
\mu^{\delta}(e) := \mu(e\cap \Omega_{\delta})$ for any Borel set $ e\subset \Omega$, 
$H_{h}$ is defined in (\ref{def_H}) and $\phi_{i,h}$ is introduced in (\ref{fem_node_func}).

\begin{theorem}
\label{thm_converge2}
Let Assumptions~\ref{assmp_boundary_para}, \ref{assmp_R}, \ref{assmp_source_integral} hold for the domain 
$\Omega$ and the function $R$. Then there is a unique function $u\in W^{+}(\Omega)$ satisfying (\ref{weak_eqs}) 
and (\ref{weak_eq3}). In addition, for any $\sigma > 0$, 
\begin{align}
\label{convergence_weak_solution}
\lim_{\delta\rightarrow 0+} \big(
\lim_{h\rightarrow 0, h\in I} \Vert u_{h}^{\delta} - u \Vert_{L^{\infty}(\overline{\Omega_{\sigma}})}\big) =0,
\end{align}where  $\Omega_{\sigma}= \{\bx\in\Omega: \text{dist}(\bx,\partial\Omega)> \sigma\}$.
\end{theorem}

\begin{proof}
For any $\delta>0$, we look for $u^{\delta}\in W^{+}(\Omega)\cap C^{0}(\overline{\Omega})$ satisfying
\begin{eqnarray}
\label{delta_eqs}
\left\{\begin{array}{lll}
\displaystyle\int_{\partial u^{\delta}(e)}R(\bp)d\bp & =& \mu^{\delta} (e)\quad \forall \text{  Borel set} \ e\subset\Omega,\\
\\
\ \hspace{60pt} u^{\delta}  &=& g\quad \text{ on } \partial\Omega.
\end{array}
\right.
\end{eqnarray}
It is easy to check that Assumption~\ref{assmp_source_decay} holds for $\mu^{\delta}$ 
with  $\lambda$ large enough. Thus, one obtains that $k<K$ where 
\begin{align*}
K = \dfrac{d+\tau+1}{\tau+2} + \dfrac{\lambda}{2}. 
\end{align*} 
Then by Theorem~\ref{thm_converge1}, there is a unique function $u^{\delta}
\in W^{+}(\Omega)\cap C^{0}(\overline{\Omega})$ satisfying (\ref{delta_eqs}). 
By \cite[Theorem~$10.4$]{Bakelman94}, it is easy to see that 
\begin{align}
\label{u_delta_bound}
\min_{\by\in \partial\Omega}g(\by) - \big(\sup_{\bx,\bx^{\prime}\in \Omega}\vert \bx - \bx^{\prime} \vert \big)
\cdot g_{R}^{-1}(\int_{\Omega}R(\bp) d\bp) \leq u^{\delta}(\bx) \leq \max_{\by\in\partial\Omega} g(\by), 
\quad \forall \bx \in \Omega.  
\end{align}
Here,
\begin{align*}
g_{R}(\rho):= \int_{B_{\rho}(\boldsymbol{0})} R(\bp) d\bp,\quad \forall \rho>0.
\end{align*}
By Theorem~\ref{thm_comparison}, for any $0<\delta^{\prime}<\delta$, it holds:
\begin{align}
\label{monotone_decay}
u^{\delta^{\prime}}(\bx) \leq u^{\delta}(\bx)\quad \forall \bx \in \Omega.
\end{align}
By (7.4)-(7.5), we know that $\displaystyle \lim_{\delta\rightarrow 0^+}u^{\delta}(\bx)$ exists for any $\bx \in \Omega$. Then we define 
\begin{align*}
u(\bx) := \lim_{\delta\rightarrow 0^+}u^{\delta}(\bx)\quad \forall \bx \in \Omega.
\end{align*}
Obviously, $u \in W^{+}(\Omega)$ and for any $\delta > 0$, it holds: $u(\bx) \leq u^{\delta}(\bx),\ \forall \bx \in \Omega.$
Since $u^{\delta}|_{\partial \Omega} = g$ for any $\delta >0$,  then we obtain
\begin{align*}
\limsup_{\Omega\ni\bx^{\prime}\rightarrow \bx} u(\bx^{\prime}) \ \ \  \leq\ \  g(\bx)
\quad \forall \bx\in \partial\Omega.
\end{align*}
By Theorem~\ref{thm_convex_sequence_converge_domain}, (\ref{u_delta_bound}) and (\ref{monotone_decay}),  we get that
\begin{align*}\left\{\begin{array}{lll}
& \displaystyle\lim_{\delta\rightarrow 0^+}\Vert u^{\delta} - u\Vert_{L^{\infty}(\overline{\Omega_{\sigma}})}\ \ =\ \ 0,
\quad \forall \sigma>0, \\ \\
&\displaystyle \int_{\partial u(e)}R(\bp)d\bp = \mu (e)\quad \forall e \text{ a Borel set of } \Omega.
\end{array}
\right.
\end{align*}
Thus $u$ satisfies (\ref{weak_eqs}). 

For any function $v\in W^{+}(\Omega)$ satisfies (\ref{weak_eqs}), we know, by Theorem~\ref{thm_comparison},  that
$
v(\bx) \leq u^{\delta}(\bx)$ for all  $ \bx\in \Omega \ \text{and} \ \delta>0,
$
which  implies that  
$v(\bx) \leq u(\bx), \forall \bx \in \Omega.$
Therefore, $u$ satisfies (1.3). 

Finally, by Applying Theorem~\ref{thm_converge1} to $u^{\delta}$ and $u_{h}^{\delta}$, 
we obtain (\ref{convergence_weak_solution}) immediately. 
\end{proof}
\vspace{20pt}

\appendix
\section{Proof of Theorem~\ref{thm_border_converge}}
\label{sec_proof_thm_border_converge}
\begin{proof}
We follow the original proof of \cite[Theorem~$10.6$]{Bakelman94}, which consists of three parts. 
In the following, we give detailed explanation of the first part, which is very geometrically intuitive.
Then, we give revision to the second part due to Assumption~\ref{assmp_R}, which is weaker than 
\cite[Assumption~$10.1$]{Bakelman94}. Finally, in the third part we explain why the revision made in 
the second part does not affect the analysis. 

\begin{itemize}
\item[Part $1$.]
Suppose that $b_{0}$ does not coincide with $\tilde{b}$ on $\partial\Omega$. 
By Lemma~\ref{lemma_border_upper_bound}, $b_{0} \leq \tilde{b}$ on $\partial\Omega$. 
Then, there is $\bx_{0} \in \partial\Omega$ such that 
\begin{align*}
b_{0}(\bx_{0}) < \tilde{b}(\bx_{0}). 
\end{align*}
Now we introduce special Cartesian coordinates in $\mathbb{R}^{d}$ and $\mathbb{R}^{d+1}$:
the axes $x^{1},\cdots, x^{d-1}$  are in the supporting $(d-1)$-dimensional plane $\alpha$ of $\partial \Omega$ 
at the point $\bx_{0}$, the axis $x^{d}$ is orthogonal to $\alpha$, and finally axis $z$ is orthogonal to the hyperplane $\mathbb{R}^{d}$. For simplicity, we take  the point $\bx_{0}\in\mathbb{R}^{d}$ 
 to be $\boldsymbol{0}\in\mathbb{R}^{d}$ and we  define two points in $\mathbb{R}^{d+1}$ by
\begin{align*}
\bQ:= (0,\cdots, 0, \tilde{b}(\boldsymbol{0}))\text{  and  } 
\overline{\bQ}:= (0,\cdots, 0, b_{0}(\boldsymbol{0}))
\end{align*}
and  for any  $0<\delta <1$,   we introduce two new points in $\mathbb{R}^{d+1}$
\begin{align*}
\bQ^{\prime}:= (0,\cdots, 0, \tilde{b}(\boldsymbol{0}) - \delta \Delta l)\text{  and  } 
\bQ^{\prime \prime}:= (0,\cdots, 0, b_{0}(\boldsymbol{0}) - \delta \Delta l),
\end{align*}
where $\Delta l := \tilde{b}(\boldsymbol{0}) - b_{0}(\boldsymbol{0})$. 
Obviously, $\bQ$, $\overline{\bQ}$, $\bQ^{\prime}$ and $\bQ^{\prime\prime}$ are all along the $z$-axis. 
Now consider two hyperplanes $\beta^{\prime}$ and $\beta^{\prime \prime}$ in $\mathbb{R}^{d+1}$ by equations: 
\begin{align*}\left\{\begin{array}{lll}
\beta^{\prime}&: \quad z = \tilde{b}(\boldsymbol{0}) - \delta \Delta l - \gamma^{-1}x^{d},\\ \\
\beta^{\prime \prime}&: \quad z = b_{0}(\boldsymbol{0}) - \delta \Delta l,
\end{array}
\right.
\end{align*}
where $\gamma$ is a sufficient small positive number. 

The task in part $1$ is to show (\ref{claim_intersect10}) holds. That is,  we need to prove that 
\begin{align*}
\int_{\partial V(\overline{\bQ})}R(\bp) d\bp \leq d_{2}\gamma^{\lambda + \frac{d+\tau+1}{\tau+2}}.
\end{align*}
Here  $V$ is the convex cone with the vertex o  $\overline{\bQ}$ and the basis   $\beta^{\prime}\cap \bK$ and $\bK$ is defined in (A.1).

We would like to point out that with respect to these three numbers $\delta, \gamma$ and $n$, 
$\delta\rightarrow 0$ implies that $\gamma\rightarrow 0$, and $\gamma\rightarrow 0$ implies 
that $n\rightarrow +\infty$ in the following analysis.
In part $1$ of the proof, we choose $0< \delta < 1$ arbitrarily, $\gamma>0$ small enough, and   $n$ large enough.

Let $Z = \partial \Omega \times \mathbb{R}\subset \mathbb{R}^{d+1}$. Then $Z$ bounds some convex body 
$\bK$ together with the hyperplanes $\beta^{\prime}$ and $\beta^{\prime\prime}$. Here the convex body $\bK$ is define by 
\begin{align}
\label{def_K}
\bK := \{ (\bx, z)\in \overline{\Omega}\times\mathbb{R}: 
b_{0}(\boldsymbol{0})-\delta \Delta l \leq z \leq \tilde{b}(\boldsymbol{0})-\delta \Delta l - \gamma^{-1}x^{d}\}.
\end{align}
It is easy to see that $\bK$ is a closed set in $\mathbb{R}^{d+1}$ and $\text{Int}(\bK)\neq \emptyset$ for any $\delta,\gamma>0$.
We define 
\begin{align*}
H(\bK):=\{\bx\in\mathbb{R}^{d} : \exists z \in \mathbb{R}\text{ such that }(\bx,z)\in \bK\},
\end{align*}
which is the projection of $\bK$ on $\mathbb{R}^{d}$. 
If $\gamma>0$ is small enough,  we know that
\begin{align*}
H(\bK) = \{\bx\in \overline{\Omega}: x^{d} \leq \bar{x}^{d} \}, \ \text{with}\ 
\bar{x}^{d}:= \gamma (\tilde{b}(\boldsymbol{0}) - b_{0}(\boldsymbol{0})). 
\end{align*}
According to Lemma~\ref{lemma_continuous_border} and part $(a)$ of Assumption~\ref{assmp_border}, 
for any $n\in\mathbb{N}$, $v_{n}$ can be extended continuously to $\partial\Omega_{n}$ such that 
$v_{n}(\bx) = b_{n}(\bx),\ \forall \bx \in \partial\Omega_{n}.$
We define
\begin{align*}
S_{v_{n}} := \{ (\bx,v_{n}(\bx)): \bx \in \overline{\Omega_{n}}\},\quad \forall n\in\mathbb{N}. 
\end{align*}
In the following, we give five important claims (\ref{claim_intersect1}) - (\ref{claim_intersect5}):

1) for any given $\gamma>0$, it holds:
\begin{align}
\label{claim_intersect1}
\text{Int}(S_{v_{n}}) \cap \text{Int}(\bK)\neq \emptyset,\text{ with} \ n \text{ large enough}.
\end{align}

In fact, due to the definition of $b_{0}$, there is $\bx\in\text{Int}(H(\bK))$ satisfying 
\begin{align*}
b_{0}(\boldsymbol{0}) - \delta \Delta l < v_{0}(\bx) < b_{0}(\boldsymbol{0}) + \frac{1}{2}(1-\delta)\Delta l.
\end{align*}
Since $b_{0}(\boldsymbol{0})+\frac{1}{2}(1-\delta)\Delta l < \tilde{b}(\boldsymbol{0})-\delta \Delta l$, 
we can choose $\bx\in\text{Int}(H(\bK))$ close enough to $\boldsymbol{0}$ such that 
\begin{align*}
b_{0}(\boldsymbol{0}) - \delta \Delta l < v_{0}(\bx) < \tilde{b}(\boldsymbol{0}) - \delta \Delta l - \gamma^{-1}x^{d}.
\end{align*}
Since $v_{n}(\bx)$ converges to $v_{0}(\bx)$ as $n\rightarrow +\infty$, then { if } $n$ is large enough, it holds:
\begin{align*}
b_{0}(\boldsymbol{0}) - \delta \Delta l < v_{n}(\bx) < \tilde{b}(\boldsymbol{0}) - \delta \Delta l - \gamma^{-1}x^{d}.
\end{align*} 
Then, $(\bx, v_{n}(\bx))\in \text{Int}(S_{v_{n}})\cap \text{Int}(\bK)$, if $n$ is large enough. Therefore the claim (\ref{claim_intersect1}) holds true.
\vspace{5pt}

2) If $ \gamma>0$  is small enough and 
$n$ \text{ is large enough}, it holds:
\begin{align}
\label{claim_intersect2}
\partial \bK \cap S_{v_{n}} \subset (\partial \bK \cap \beta^{\prime})\backslash Z.
\end{align}

In fact, we can  easily see that
$\partial\bK = \Gamma \cup \tilde{H}(\bK)\cup (\partial\bK \cap \beta^{\prime}), $ where
\begin{align*}\left\{\begin{array}{lll}
 \Gamma\ :=\ \{(\bx, z)\in \mathbb{R}^{d+1}:\bx \in \partial H(\bK),\
b_{0}(\boldsymbol{0}) - \delta \Delta l < z < \tilde{b}(\boldsymbol{0}) - \delta \Delta l - \gamma^{-1}x^{d} \},\\ \\
\tilde{H}(\bK)\ := \ \{ (\bx, b_{0}(\boldsymbol{0})-\delta \Delta l): \bx \in H(\bK) \}.
\end{array}
\right.\end{align*}If $\gamma > 0$ is small enough, then $\Omega\cap \{\bx \in \mathbb{R}^{d}: x^{d} = \bar{x}^{d} \}\neq \emptyset$. 
Then we have 
\begin{align*}\left\{\begin{array}{lll}
& H(\bK) = \{\bx \in \overline{\Omega}: x^{d}\leq \bar{x}^{d} \},\
\tilde{H}(\bK)=\{(\bx, b_{0}(\boldsymbol{0})-\delta \Delta l): x\in \overline{\Omega},\quad x^{d}\leq \bar{x}^{d} \},\\ \\
& \Gamma =\{(\bx, z)\in \mathbb{R}^{d+1}:\bx \in \partial\Omega, \  x^{d} < \bar{x}^{d}, \
b_{0}(\boldsymbol{0}) - \delta \Delta l < z < \tilde{b}(\boldsymbol{0}) - \delta \Delta l - \gamma^{-1}x^{d} \},\\ \\
& \overline{\Gamma} =\{(\bx, z)\in \mathbb{R}^{d+1}:\bx \in \partial\Omega, \  x^{d} \leq \bar{x}^{d}, \
b_{0}(\boldsymbol{0}) - \delta \Delta l \leq z \leq \tilde{b}(\boldsymbol{0}) - \delta \Delta l - \gamma^{-1}x^{d} \}.
\end{array}
\right.
\end{align*}

The claim (A.2) follows directly if we can prove that $S_{v_{n}}\cap \overline{\Gamma}=S_{v_{n}}\cap \tilde{H}(\bK)=\emptyset$ if $\gamma>0$ is small enough 
and $n$ is large enough.

Firstly, we show that $S_{v_{n}}\cap \overline{\Gamma}=\emptyset$ if $\gamma>0$ is small enough 
and $n$ is large enough. If not, then  there is a subsequence of $\mathbb{N}$, still denote by $\mathbb{N}$ , 
for simplicity, such that 
$S_{v_{n}}\cap \overline{\Gamma}\neq\emptyset$ for any $n\in\mathbb{N}$. 
Then for any $n\in\mathbb{N}$, there is $\bx_{n}\in \partial\Omega_{n}\cap \Omega$, such that 
$
x_{n}^{d}\leq \bar{x}^{d},$ and $b_{0}(\boldsymbol{0}) - \delta \Delta l \leq v_{n}(\bx_{n})
\leq \tilde{b}(\boldsymbol{0}) - \delta \Delta l - \gamma^{-1}x_{n}^{d}. 
$
According to Assumption~\ref{assmp_boundary_para}, $\lim_{\gamma\rightarrow 0}^{T} H(\bK) 
= \boldsymbol{0}\in \mathbb{R}^{d}$.  From $(b)$ of Assumption~\ref{assmp_border},  we obtain 
\begin{align*}
\sup_{\bx \in\partial\Omega, \ x^{d}\leq\bar{x}^{d}}\vert \tilde{b}(\bx) - \tilde{b}(\boldsymbol{0}) \vert {\longrightarrow} 0, \ \text{as}\ \gamma\rightarrow0.
\end{align*}
Hence we can choose $\gamma>0$ small enough, such that  it holds:
\begin{align*}
\vert \tilde{b}(\bx) - \tilde{b}(\boldsymbol{0}) \vert < \frac{1}{3}\delta \Delta l,
\quad \forall \bx \in \{ \by\in\partial\Omega: y^{d} \leq \bar{x}^{d} \}.
\end{align*}
Since $\bx_{n}\in \{ \by \in \partial\Omega: y^{d} \leq \bar{x}^{d}\}$, then we get
$$\vert \tilde{b}(\bx_{n}) - \tilde{b}(\boldsymbol{0}) \vert < \frac{1}{3}\delta \Delta l,
\quad \forall n\in \mathbb{N}. $$
According to part $(b)$ of Assumption~\ref{assmp_border}, there is a subsequence of $\mathbb{N}$ 
which we still denote by $\mathbb{N}$ for the sake of simplicity, such that 
\begin{align*}
(\bx_{n}, v_{n}(\bx_{n}))=  (\bx_{n}, b_{n}(\bx_{n}))\longrightarrow
(\tilde{\bx}, \tilde{b}(\tilde{\bx})), \ \text{as}\ n\rightarrow+\infty
\end{align*}
for some point $\tilde{\bx}\in \{\by\in\partial\Omega: y^{d}\leq \bar{x}^{d}\}$. 
Then if $n\in\mathbb{N}$ large enough, 
\begin{align*}
\vert v_{n}(\bx_{n}) - \tilde{b}(\boldsymbol{0}) \vert < \frac{2}{3}\delta \Delta l,
\end{align*}which implies that if $n\in\mathbb{N}$ large enough, 
\begin{align*}
v_{n}(\bx_{n}) > \tilde{b}(\boldsymbol{0}) - \frac{2}{3}\delta \Delta l 
> \tilde{b}(\boldsymbol{0}) - \delta \Delta l - \gamma^{-1}x_{n}^{d}.
\end{align*}
This is a contradiction. Thus $S_{v_{n}}\cap \overline{\Gamma} = \emptyset$ for $\gamma>0$ small enough and $n$ large enough. 

Secondly, we show that $S_{v_{n}}\cap \tilde{H}(\bK)=\emptyset$ if $\gamma>0$ is small enough and $n$ 
is large enough. According to Assumption~\ref{assmp_boundary_para}, we know that $\lim_{\gamma\rightarrow 0}^{T} H(\bK) 
= \boldsymbol{0}\in \mathbb{R}^{d}$.  By the definition of $b_{0}$,  we can see that if $\gamma>0$ is small enough, 
\begin{align*}
v_{0}(\bx) > b_{0}(\boldsymbol{0}) - \frac{1}{6}\delta \Delta l, \quad \forall \bx \in \text{Int}(H(\bK))
\end{align*}
and there is a point $\tilde{\bx}\in H(\bK)\cap \Omega$ such that 
\begin{align*}
\vert v_{0}(\tilde{\bx}) - b_{0}(\boldsymbol{0}) \vert < \frac{1}{6}\delta \Delta l. 
\end{align*}
Since $\bar{x}^{d}:= \gamma (\tilde{b}(\boldsymbol{0}) - b_{0}(\boldsymbol{0}))$ and  $H(\bK) = \{\bx \in \overline{\Omega}: x^{d}\leq \bar{x}^{d} \}$ for $\gamma>0$  small enough, we can choose $\tilde{\bx}$ satisfying $\tilde{x}^{d} = \bar{x}^{d}$ if $\gamma>0$ is small enough. 
Obviously, $\tilde{\bx} \in \text{Int}(\{\bx\in \partial H(\bK): x^{d} = \bar{x}^{d} \})$.
We define 
\begin{align*}
E(\bK):= \{ \bx \in H(\bK): & \exists \by \in \partial\Omega\text{ with } y^{d}\leq \bar{x}^{d} 
\text{ such that } \bx \text{ is in the line segment} \\
& \text{ between } \by \text{ and } \tilde{\bx}, 
\text{ and } \text{dist}(\bx, \tilde{\bx}) \leq \frac{1}{2} \text{dist}(\by, \tilde{\bx}) \}. 
\end{align*}
It is easy to see that $E(\bK)$ is a closed subset of $\Omega$ and  $\text{dist}(\partial\Omega, E(\bK))>0$. 
From \cite[Lemma~$3.1$]{Bakelman94}, (\ref{assmp_subdomains}) and part $(a)$ of Assumption~\ref{assmp_function_decay},  we get
\begin{align*}
\lim_{n\rightarrow +\infty}\Vert v_{n} - v_{0} \Vert_{L^{\infty}(E(\bK))}= 0,
\end{align*}
which implies that  if $n$ is large enough, 
\begin{align*}
v_{n}(\tilde{\bx}) <  b_{0}(\boldsymbol{0})+\frac{1}{4}\delta \Delta l,\ \text{and}\ 
v_{n}(\bx) >  b_{0}(\boldsymbol{0}) -\frac{1}{4}\delta \Delta l, \ \forall \bx \in E(\bK). 
\end{align*}
Due to the fact that $v_{n}\in W^{+}(\Omega_{n})$ and the definition of $E(\bK)$, we know that if $n$ is large enough, 
\begin{align*}
v_{n}(\bx) > b_{0}(\boldsymbol{0}) - \frac{3}{4}\delta \Delta l,\quad \forall \bx \in H(\bK)\cap \Omega_{n}.
\end{align*}
Thus if $n$ is large enough, 
$v_{n}(\bx) > b_{0}(\boldsymbol{0}) - \delta \Delta l,\quad \forall \bx \in H(\bK)\cap \overline{\Omega_{n}}$,  which shows that $S_{v_{n}}\cap \tilde{H}(\bK)=\emptyset$ if $\gamma>0$ is small enough and $n$ is large enough. 
\vspace{5pt}

3)If $\gamma>0$ is small enough and $n\in\mathbb{N}$ is large enough,
\begin{align}
\label{claim_intersect3}
\exists \bQ_{n}\in \text{Int}\bK \cap \textbf{Int}(S_{v_{n}}) \text{ such that }
\lim_{n\rightarrow+\infty}\text{dist}(\overline{\bQ}, \bQ_{n})\ =\ 0.
\end{align}

In fact, $\forall\epsilon>0$, by  the definitions of $b_{0}$ and $H(\bK)$, 
there is $\bx\in \text{Int}(H(\bK))$ such that $\vert \bx\vert< \epsilon$ and there holds:
\begin{align*}
 -\frac{1}{2}\min (\epsilon, \delta \Delta l) < v_{0}(\bx) -b_{0}(\boldsymbol{0})
< \frac{1}{2}\min (\epsilon, (1-\delta)\Delta l - \gamma^{-1}x^{d}).
\end{align*}
By part $(a)$ of Assumption~\ref{assmp_function_decay}, for any $n\in \mathbb{N}$ large enough, 
\begin{align*}
 -\min (\epsilon, \delta \Delta l) < v_{n}(\bx) -b_{0}(\boldsymbol{0})
< \min (\epsilon, (1-\delta)\Delta l - \gamma^{-1}x^{d}).  
\end{align*}
Thus for any $n\in\mathbb{N}$ large enough, we can infer 
\begin{align*}
\text{dist}((\bx,v_{n}(\bx)), \overline{\bQ})< \epsilon\  \text{and} \ (\bx ,v_{n}(\bx))\in \text{Int}(K).
\end{align*}
By (\ref{assmp_subdomains}), $\bx\in \Omega_{n}$ if $n\in\mathbb{N}$ is large enough. 
Therefor, (\ref{claim_intersect3}) holds true. \\

4) If $\gamma>0$ is small enough and $n\in\mathbb{N}$ is large enough,
\begin{align}
\label{claim_intersect4}
\partial V_{n}(\bQ_{n}) \subset \partial v_{n}(\text{Int}(H_{n}(\bK))),
\end{align}
where $\bQ_{n}$ is defined in (\ref{claim_intersect3}) and 
\begin{align*}\left\{\begin{array}{lll}
S_{n}(\bK) :=  S_{v_{n}}\cap \bK,\quad \beta^{\prime}(\bK) := \beta^{\prime}\cap \bK, \\ \\
H_{n}(\bK):=  \{\bx \in H(\bK): \exists z\in \mathbb{R}\text{ such that }(\bx, z)\in S_{n}(\bK)\},\\ \\
 V_{n}\ \text{is the convex cone with the vertex}\  \bQ_{n}\ \text{ and the base } \ \beta^{\prime}(\bK).
\end{array}
\right.
\end{align*}

In fact, by (\ref{claim_intersect2}),  we see that
\begin{align*}
H_{n}(\bK) = \{ \bx \in H(\bK)\cap \overline{\Omega_{n}}:
 \tilde{b}(\boldsymbol{0}) -\delta \Delta l - \gamma^{-1}x^{d} \geq v_{n}(\bx)\}.
\end{align*}for $\gamma>0$ small enough and $n\in\mathbb{N}$  large enough.

We define functions $\tilde{V}_{n}$ in $H(\bK)$ by
\begin{align*}
\tilde{V}_{n}(\bx) = \inf_{z\in \mathbb{R}}(\bx, z)\in V_{n},\quad \forall \bx \in H(\bK). 
\end{align*}
Then $\tilde{V}_{n}\in W^{+}(H(\bK))$.
Furthermore, by (\ref{claim_intersect2}, \ref{claim_intersect3}), we can see that $
 \bQ_{n}\in S_{n}(\bK)\cap V_{n},$ and $\tilde{V}_{n}(\bx) \leq v_{n}(\bx),$ $ \forall \bx \in \partial H_{n}(\bK)$ if $\gamma>0$ is small enough and $n\in\mathbb{N}$ is large enough.

Let $T$ be a supporting hyperplane of $V_{n}$ at $\bQ_{n}$ with the equation 
\begin{align*}
z = z_{T} + \bp_{T}\cdot \bx. 
\end{align*}
Let $(\bx_{n}, z_{n}) = \bQ_{n}$ where $\bx_{n}\in \mathbb{R}^{d}$. Then  we obtain
\begin{align*}\left\{\begin{array}{lll}
& v_{n}(\bx_{n}) = z_{T} + \bp_{T}\cdot\bx_{n},\\ \\
& v_{n}(\bx) \geq z_{T} + \bp_{T}\cdot \bx,\quad \forall \bx \in \partial H_{n}(\bK).
\end{array}
\right.
\end{align*}
By (\ref{claim_intersect3}),  we know that
$\tilde{b}(\boldsymbol{0}) - \delta \Delta l - \gamma^{-1}x_{n}^{d} > z_{n}=v_{n}(\bx_{n}).$
Thus $\bx_{n}\in \text{Int}(H_{n}(\bK))$.  In the following, we shall show that $\bp_{T}\in \partial v_{n}(\text{Int}(H_{n}(\bK)))$. In fact, 
if $v_{n}(\bx) \geq z_{T}+\bp_{T}\cdot\bx,\quad \forall \bx \in H_{n}(\bK),$ then $\bp_{T}\in \partial v_{n}(\text{Int}(H_{n}(\bK)))$. 
On the other hand, if 
$
\{\bx \in H_{n}(\bK): v_{n}(\bx) < z_{T}+\bp\cdot\bx\}\neq \emptyset,
$
then by the fact
\begin{align*}
v_{n}(\bx) \geq z_{T} + \bp_{T}\cdot \bx,\quad \forall \bx \in \partial H_{n}(\bK),
\end{align*}
we know that 
\begin{align*}
\{\bx \in H_{n}(\bK): \tilde{V}_{n}(\bx) < z_{T}+\bp_{T}\cdot\bx\} \subset \text{Int}H_{n}(\bK). 
\end{align*}
By \cite[Lemma~$1.4.1$]{Gutierrez01},   $\bp_{T}\in \partial v_{n}(\text{Int}(H_{n}(\bK)))$. 
Thus (\ref{claim_intersect4}) holds true. 
\vspace{5pt}

5) If $\gamma>0$ is small enough and $n\in\mathbb{N}$ is large enough
\begin{align}
\label{claim_intersect5}
\int_{\partial V(\overline{\bQ})} R(\bp) d\bp\leq \liminf_{n\rightarrow +\infty}\int_{\partial V_{n}(\bQ_{n})}R(\bp)d\bp,
\end{align}
where $V$ is a convex cone, with the vertex  $\overline{\bQ}$ and the basis $\beta^{\prime}(\bK)$.  

In fact, since $\text{Int}(H_{n}(\bK)) \subset H(\bK)\cap \Omega_{n}$, 
then 
\begin{align*}
\int_{\partial v_{n}(\text{Int}(H_{n}(\bK)))}R(\bp) d\bp \leq \int_{\partial v_{n}(H(\bK)\cap \Omega_{n})}R(\bp) d\bp,
\end{align*}
which,  by claim (\ref{claim_intersect4}), shows that
\begin{align}
\label{claim_intersect6}
\int_{\partial V_{n}(\bQ_{n})}R(\bp) d\bp \leq \int_{\partial v_{n}(H(\bK)\cap \Omega_{n})}R(\bp) d\bp.
\end{align}

We know that 
\begin{align*}\left\{\begin{array}{lll}
\displaystyle\int_{\partial V_{n}(\bQ_{n})}R(\bp)d\bp = \int_{\mathbb{R}^{d}}\chi_{\partial V_{n}(\bQ_{n})}(\bp)R(\bp)d\bp, \\ \\
\displaystyle\int_{\partial V\overline{\bQ})}R(\bp)d\bp = \int_{\mathbb{R}^{d}}\chi_{\partial V\overline{\bQ})}(\bp)R(\bp)d\bp,
\end{array}
\right.
\end{align*}
where $\chi_{\partial V_{n}(\bQ_{n})}$ and $\chi_{\partial V(\overline{\bQ}))}$ are characteristic functions. 
According to Fatou's Lemma, 
\begin{align*}
\int_{\mathbb{R}^{d}} \liminf_{n\rightarrow + \infty}\big( \chi_{\partial V_{n}(\bQ_{n})}(\bp) R(\bp)\big)d\bp 
\leq \liminf_{n\rightarrow +\infty}\int_{\mathbb{R}^{d}}\chi_{\partial V_{n}(\bQ_{n})}(\bp) R(\bp)d\bp. 
\end{align*}
Hence,  to prove (\ref{claim_intersect6}), it is sufficient to show that 
\begin{align}
\label{claim_intersect7}
\liminf_{n\rightarrow +\infty}\chi_{\partial V_{n}(\bQ_{n})}(\bp)  \geq \chi_{\partial V (\overline{\bQ}))}(\bp),
\quad \forall \text{ a.e. }  \bp\in\mathbb{R}^{d}.
\end{align}
Let $T$ be a supporting hyperplane of $V$ at $\overline{\bQ}$ with 
\begin{align*}
z = z_{T} + \bp_{T}\cdot \bx. 
\end{align*}
We notice that the projection of $\beta^{\prime}(\bK)$ onto $\mathbb{R}^{d}$ is $H(\bK)$ 
and $\overline{\bQ}$ is the vertex of the convex cone $V$. Then we know that 
\begin{align*}
\bp_{T} \in \text{Int}\big( \partial V (\overline{\bQ}) \big) \Longleftrightarrow
\tilde{b}(\boldsymbol{0}) - \delta \Delta l - \gamma^{-1}x^{d} > 
z_{T} + \bp_{T}\cdot \bx,\quad \forall \bx \in \partial H(\bK). 
\end{align*}
We define 
\begin{align*}
\epsilon_0 :=\inf_{\bx \in \partial H(\bK)} \big( \big( \tilde{b}(\boldsymbol{0}) - \delta \Delta l - \gamma^{-1}x^{d} \big) 
- \big( z_{T} + \bp_{T}\cdot \bx\big) \big).
\end{align*}
Then $\epsilon_0>0$.
For any $n\in\mathbb{N}$, we choose $z_{n}\in\mathbb{R}$ such that 
\begin{align*}
z = z_{n} + \bp_{T}\cdot \bx
\end{align*}
is a hyperplane passing through $\bQ_{n}$. 
By (\ref{claim_intersect3}) and the fact that $\epsilon_0>0$,  
it holds:
\begin{align*}
\tilde{b}(\boldsymbol{0}) - \delta \Delta l - \gamma^{-1}x^{d} \geq 
z_{n} + \bp_{T}\cdot \bx,\quad \forall \bx \in \partial H(\bK). 
\end{align*}if $\gamma>0$ is small enough and $n\in\mathbb{N}$ is large enough.
Then $\bp_{T}\in \partial V_{n}(\bQ_{n})$,  which shows that (\ref{claim_intersect7}) holds. 
Thus (\ref{claim_intersect5}) is true. 
\vspace{15pt}

In the following, we finish the proof of part 1. According to Assumption~\ref{assmp_boundary_para}, if $\gamma>0$ is small enough, the Borel set $H(\bK) \subset  U_{\rho}(\boldsymbol{0}) \cap \Omega. $
By part $(b)$ of Assumption~\ref{assmp_function_decay}, 
\begin{align*}
\liminf_{n\rightarrow + \infty} \int_{\partial v_{n}(H(\bK)\cap \Omega_{n})}R(\bp) d\bp 
\leq C_{1}\big( \sup_{\bx \in H(\bK)} \text{dist}(\bx , \partial\Omega) \big)^{\lambda} \vert H(\bK) \vert.
\end{align*}
The last inequality with (\ref{claim_intersect5}) implies the inequality
\begin{align}
\label{claim_intersect8}
\int_{\partial V(\overline{Q})}R(\bp) d\bp \leq 
C_{1}\big( \sup_{\bx \in H(\bK)} \text{dist}(\bx , \partial\Omega) \big)^{\lambda} \vert H(\bK) \vert.
\end{align}
Clearly we have
\begin{align*}
\sup_{\bx \in H(\bK)} \text{dist}(\bx , \partial\Omega) = \gamma \Delta l.
\end{align*}

Let $P :=\{\bx\in \mathbb{R}^{d}: \eta(\boldsymbol{0}) \big( \Sigma_{i=1}^{d-1} \vert x^{i} \vert^{2}\big)^{\frac{\tau + 2}{2}}
\leq x^{d} \leq \gamma \Delta l\},
$
where $\eta(\boldsymbol{0})$ is the positive constant introduced in Definition~\ref{def_local_para}. Hence 
$H(\bK) \subset P$
and it holds:
\begin{align}
\label{claim_intersect9}
\vert P\vert &= m_{d-1}\int_{0}^{\gamma\Delta l} \big( \dfrac{l}{\eta(\boldsymbol{0})} \big)^{\frac{d-1}{\tau+2}}dl \nonumber\\
&= \frac{\tau+2}{d+1}m_{d-1} (\eta(\boldsymbol{0}))^{-\frac{d-1}{\tau+2}}
\big(\gamma \Delta l\big)^{\frac{d+\tau+1}{\tau+2}}=d_{1}\gamma^{\frac{d+\tau+1}{\tau+2}}.
\end{align}
From (\ref{claim_intersect8})- (\ref{claim_intersect9}), we know that if $\gamma>0$ is small enough,  
\begin{align}
\label{claim_intersect10}
\int_{\partial V(\overline{\bQ})}R(\bp) d\bp \leq d_{2}\gamma^{\lambda + \frac{d+\tau+1}{\tau+2}}, 
\end{align}
where $d_{2} = C_{1} d_{1}(\Delta l)^{\lambda}$ and $d_{1}$ is a positive constant depending only on 
given constants $0\leq \tau <+\infty$, $\eta(\boldsymbol{0})>0$, $\Delta l$ and $m_{d-1}$. 
(Notice that $\tau$ and $\eta(\boldsymbol{0})$ are introduced in (\ref{def_local_para}), 
$m_{d-1}$ is the volume of the unit $(d-1)$-ball)
\\
\\

\item[Part $2$.]
According to \cite[($10.45$)]{Bakelman94} and \cite[Lemma~$10.4$]{Bakelman94}, 
We know that 
\begin{align}
\label{lower_bound1}
\int_{\partial V(\overline{\bQ})} R(\bp) d\bp \geq \int_{H^{d}} R(\bp) d\bp.
\end{align}
Here, $H^{d}$ is the $d$-dimensional cone of revolution with axis $p^{d}$ (axis in $\mathbb{R}^{d}$), 
vertex $(0,\cdots,0,-\frac{\delta}{\gamma})\in \mathbb{R}^{d}$ and base $H^{d-1}$, which is the $(d-1)$-dimensional ball given by the following equations:
$
\vert p^{1}\vert^{2}+\cdots + \vert p^{d-1} \vert^{2} \leq (C^{\prime})^{2} \gamma^{-\frac{2}{\tau + 2}}, 
\  p^{d} = -C^{\prime \prime} \gamma^{-1}
$
 and the constants $C^{\prime}$ and $C^{\prime\prime}$ are given by 
\begin{align*}
C^{\prime} = \frac{\tau+2}{\tau+1}(1-\delta) ( \Delta l ) ^{\frac{\tau + 2}{\tau + 1}}  (\eta(\boldsymbol{0}))^{\frac{1}{\tau + 2}},\ 
 C^{\prime\prime} = \frac{\tau + 2 - \delta}{\tau + 1}.
\end{align*}
Obviously, $C^{\prime}$ and $C^{\prime\prime}$ do not depend on $\gamma$ and have positive limits as $\delta \rightarrow 0$. 

The convex cone $H^{d}$ lies between two parallel hyperplanes in $\mathbb{R}^{d}$
\begin{align*}
p^{d} = -\delta \gamma^{-1},\quad p^{d} = -\frac{\tau + 2 -\delta}{(\tau + 1)\gamma}.  
\end{align*}
In the original proof of \cite[Theorem~$10.6$]{Bakelman94}, 
\begin{align}
\label{lower_bound2}
\int_{H^{d}} R(\bp) d\bp \geq C_{0}\int_{H^{d}} \vert \bp \vert^{-2k} d\bp,
\end{align}
due to \cite[Assumption~$10.1$]{Bakelman94}. However, \cite[Assumption~$10.1$]{Bakelman94} 
is so restrictive that the Guassian curvature equation does not satisfy. Hence we use 
Assumption~\ref{assmp_R} instead. 

The revision we made is to require two positive parameters $\delta$ and $\gamma$ satisfying
\begin{align}
\label{delta_gamma}
\gamma \leq r_{0}^{-1}\delta
\end{align}
where $r_{0}$ is the positive constant defined in Assumption~\ref{assmp_R}. 

If (\ref{delta_gamma}) is satisfied, then by (\ref{lower_bound1}, \ref{lower_bound2}), we know that 
\begin{align}
\label{lower_bound3}
\int_{\partial V(\overline{\bQ})} R(\bp) d\bp \geq \int_{H^{d}} R(\bp) d\bp \geq C_{0}\int_{H^{d}} \vert \bp \vert^{-2k} d\bp. 
\end{align}
Thus the remaining part of part $2$ is the same as the second part of the proof of \cite[Theorem~$10.6$]{Bakelman94} 
(right below the proof of \cite[Lemma~$10.4$]{Bakelman94}). 
\\
\\

\item[Part $3$.]
We completely follow the third part of the proof 
of \cite[Theorem~$10.6$]{Bakelman94}.
The task in part $3$ is to show inequalities based on (\ref{lower_bound3}) does not hold 
as $\gamma >0$ approaches to zero.  We only need to verify inequalities 
\cite[($10.68$), ($10.69$), ($10.74$), ($10.77$), ($10.78$)]{Bakelman94} 
hold if (\ref{delta_gamma}) is satisfied.  Thus the third part of the proof of \cite[Theorem~$10.6$]{Bakelman94} 
can go through completely  if (\ref{delta_gamma}) is satisfied. 
Therefore the proof is complete.  

\end{itemize}
\end{proof}

\end{document}